\documentclass[a4paper,12pt]{amsart}

\usepackage{amsmath,amssymb,amsthm,graphicx,stmaryrd,bbm}
\usepackage{verbatim}
\usepackage{hyperref}
\usepackage[normalem]{ulem}
\usepackage{xcolor}

\newcommand{\NN}{\mathbb{N}}
\newcommand{\PP}{\mathbb{P}} 
\newcommand{\RR}{\mathbb{R}}

\newcommand{\ZZ}{\mathbb{Z}}

\newcommand{\cB}{\mathcal{B}} 

\newcommand{\cC}{\mathcal{C}}
 
\newcommand{\cF}{\mathcal{F}}

\newcommand{\cG}{\mathcal{G}}

\renewcommand{\a}{\alpha}
 
\renewcommand{\d}{\delta} 
\newcommand{\G}{\Gamma}

\newcommand{\g}{\gamma} 
\renewcommand{\L}{\Lambda}

\renewcommand{\l}{\lambda}
\renewcommand{\b}{\beta} 
\renewcommand{\k}{\kappa} 
\newcommand{\Om}{\Omega}
\newcommand{\om}{\omega} 
 
\newcommand{\s}{\sigma}

\newcommand{\eps}{\varepsilon}

\renewcommand{\max}{\mathrm{max}}

\newcommand{\ST}{\mathrm{ST}}
\renewcommand{\b}{\beta}
\newcommand{\oo}{\infty}

\newcommand{\es}{\varnothing}
\newcommand{\se}{\subseteq}
\newcommand{\ul}{\underline}

\newcommand{\ol}{\overline}

\newcommand{\crit}{\mathrm{c}} 
\newcommand{\perc}{\mathrm{p}} 
\newcommand{\ui}{\ol\nu}

\newcommand{\bac}{Z^\leftarrow}
\newcommand{\ddcp}{{\textsc{ddcp}}}

\newcommand{\dfl}{\Lambda}
\newcommand{\dfh}{H}

\newcommand{\ki}{\kappa}
\newcommand{\kii}{\tilde\kappa}
\newcommand{\hi}{h}
\newcommand{\hii}{\tilde h}
\newcommand{\li}{\lambda}
\newcommand{\lii}{\tilde\lambda}
\newcommand{\kre}{\kappa^\star}

\newcommand{\one}{\hbox{\rm 1\kern-.27em I}}

\newtheoremstyle{slthm}% theoremstyle with slanted body type
     {}%      Space above (empty=\topspace)
     {\baselineskip}%      Space below 
     {\slshape}%         Body font
     {\parindent}%    Indent amount (\parindent = para indent)
     {\scshape}% Thm head font
     {.}%        Punctuation after thm head
     { }%     Space after thm head: " " = normal interword space
     {}%         Thm head spec (can be left empty, meaning `normal')

\newtheoremstyle{rk}% theoremstyle with upright body type
     {}%      Space above (empty=\topspace)
     {\baselineskip}%      Space below 
     {\normalfont}%         Body font
     {\parindent}%    Indent amount (\parindent = para indent)
     {\scshape}% Thm head font
     {.}%        Punctuation after thm head
     { }%     Space after thm head: " " = normal interword space
     {}%         Thm head spec (can be left empty, meaning `normal')

\theoremstyle{slthm}
\newtheorem{definition}{Definition}[section]
\newtheorem{theorem}[definition]{Theorem}
\newtheorem{proposition}[definition]{Proposition}
\newtheorem{lemma}[definition]{Lemma}
\newtheorem{corollary}[definition]{Corollary}

\newtheorem{assumption}[definition]{Assumption}

\theoremstyle{rk}
\newtheorem{remark}[definition]{Remark}

\numberwithin{equation}{section}

\date{September 10, 2014}
\title[Percolation in contact processes]
{Sharpness versus robustness \\of the percolation transition\\
 in 2d contact processes}
 
\author{J.\ van den Berg}
\address{Centrum Wiskunde \& Informatica (CWI),
P.O.~Box 94079, 1090~GB Amsterdam, The Netherlands,
and {V}{U} University Amsterdam, The Natherlands}
\email{J.van.den.Berg@cwi.nl}
\author{J. E. Bj\"ornberg}
\address{Uppsala Universitet,  Department of Mathematics,
P.O.~Box 480,
751~06 Uppsala,
Sweden}
\email{jakob@math.uu.se}
\author{M. Heydenreich}
\address{Mathematisch Instituut,
Universiteit Leiden,
P.O.~Box 9512,
2300 RA Leiden,
The Netherlands; 
Centrum Wiskunde \& Informatica (CWI),
P.O.~Box 94079, 1090~GB Amsterdam, The Netherlands}
\email{markus@math.leidenuniv.nl}
\thanks{The research of MH is supported by the Dutch Organization of Scientific Research (NWO).
The research of JEB is supported by the Knut and Alice Wallenberg
Foundation.}

 

\begin{document}
\begin{abstract}
We study versions of the contact process with three
states, and with infections occurring at a rate
depending on the overall infection density. 
Motivated by a model described in \cite{KefiRietkAladoPueyo07} 
for vegetation patterns in arid landscapes,
we focus on percolation under  
invariant measures of such processes. 
We prove that the percolation 
transition is \emph{sharp} (for one of our models this 
requires a reasonable assumption). 
This is shown to contradict a form of 
`robust critical behaviour' 
with power law cluster size distribution for a range
of parameter values, as suggested in \cite{KefiRietkAladoPueyo07}.
\end{abstract}
\maketitle

%%%%%%%%%%%%%%%%%%%%%%%%%%%%%%%%%%%%%%%%%%%%%%%%%%%%%%%%%%%%%%%%%%%%

\section{Introduction and background}
\label{intro_sec}
Percolative systems with weak dependencies, 
such as the contact process and its variants, 
are in the spotlight of recent mathematical research.
The present article studies versions of the 
two-dimensional contact process that 
are motivated by models for vegetation patterns 
in arid landscapes, as put forward by biologists 
and agricultural researchers \cite{KefiRietkAladoPueyo07}. 
The central question we address is whether or not 
the percolation transition for these modified contact processes 
is \emph{sharp}. 
This paper demonstrates the  applicability of the sharpness
techniques for two-dimensional systems, even to quite realistic
models, and provides a rigorous mathematical basis
for the discussion in the life-science literature
\cite{anti-kefi-comment,KefiRietkAladoPueyo07,kefi11,anti-kefi,anti-kefi-reply}.

The models and questions we consider are defined 
precisely in the following
subsections.  Briefly speaking, however, 
we consider two main modifications of 
the `standard' two-dimensional contact process.
Firstly, rather than two states 0 and 1, we allow three states: $-1$,
0 and 1.  Secondly, the transition rates are allowed to
vary with the overall density
of 1's in the process itself.  Contact processes with three states
have been considered by several authors 
before, e.g.~\cite{Neuha92,Remen08}.
We consider two different types of 3-state contact
processes, one of which is closely related to the
process in~\cite{Remen08}, the other of which has not
previously appeared in the mathematical literature.
The second modification has, to the best of our knowledge,
not been considered previously in the mathematical
literature;
we call such processes `density-driven' (see Definition~\ref{ddcp_def})
and prove their existence in Section~\ref{ddcp_sec}.

Our main focus is the question of \emph{percolation} in such
processes:  whether or not, under an invariant distribution of the
process, there can be an unbounded connected set of 1's. 
(For general information on percolation we 
refer to \cite{Gr99} and \cite{BoRi06}.) 
For certain
parameter values an unbounded connected set of 1's occurs with
positive probability and for others not.  As the parameters are
varied, one obtains in this sense a phase transition which we refer to
as the \emph{percolation transition}.  
In~\cite{KefiRietkAladoPueyo07} it is suggested,
based on numerical simulation, that the type of model we consider may
exhibit a form of `robust critical behaviour', different from the
usual `sharp' phase transition in standard percolation models.  We
critically discuss this suggestion, based on rigorous results about
the percolation transition.  Our main results, on sharpness of the
transition and lack of `robustness', are stated in 
Theorems~\ref{nonrobust_thm_3}, \ref{nonrobust_thm},
and \ref{3st_sharp_thm} below.

The contact process is one of the most-studied interacting 
particle systems, see e.g.\ \cite{Ligge99} and references therein. 
Several multi-type variants have been studied;  most of them
have been proposed as models in 
theoretical biology, and  focus has typically been on the 
survival versus extinction of species. 
See, for example, 
Cox--Schinazi~\cite{CoxSchin09},
Durrett--Neuhauser~\cite{DurreNeuha97},
Durrett--Swindle~\cite{DurreSwind91}, 
Konno--Schinazi--Tanemura~\cite{KonnoSchinTanem04}, 
Kuczek~\cite{Kucze89}, 
Neuhauser~\cite{Neuha92}.
The question of percolation under invariant distributions
of the contact process was first studied 
by Liggett and Steif~\cite{LiggeSteif06}, and a sharpness result
for percolation under such distributions was first proved
by van den Berg~\cite{Berg11}, 
using some of the techniques introduced for Voronoi percolation in
\cite{BolloRiord06}.

We begin by describing the type of model we 
consider in more detail. 

\subsection{Contact processes}
\label{cp_sec}
The ordinary contact process on $\ZZ^d$ 
is a Markov process with state space 
$\{0,1\}^{\ZZ^d}$.  Elements $x\in\ZZ^d$ are  
called `sites' or `individuals'.
An element of
$\{0,1\}^{\ZZ^d}$ is typically  denoted by
$\eta=(\eta_x:x\in\ZZ^d)$
and those $x\in\ZZ^d$ for which $\eta_x=1$
are typically called `infected'.  Infected
individuals recover at rate $\k$, independently of each other
(often $\k$ is set to 1).  Alternatively, a healthy site can become
infected by an infected neighbour at rate $\l$.  This occurs
independently for different sites and independently of the
recoveries.  

The main facts about the contact process are the following. 
For any $\l\geq0$
there exists an \emph{upper invariant measure} 
$\ui$ on $\{0,1\}^{\ZZ^d}$ which can be obtained 
as the limiting distribution when initally
\emph{all} sites are infected (this follows from 
standard monotonicity arguments \cite{Ligge99}).
For each $\k>0$ there is a critical value
$\l_\crit=\l_\crit(\k)\in(0,\oo)$ such that if $\l\leq\l_\crit$ then
the process `dies out'.  This means that 
the only stationary distribution of the process
is the point mass $\d_\es$ on the configuration consisting of all
zeros, or equivalently $\ui=\d_\es$.
On the other hand, if $\l>\l_\crit$ then there is positive
chance that infection is transmitted indefinitely, and hence  $\ui\neq\d_\es$.  
In this regime, there is more than one invariant distribution, each invariant
distribution being a convex combination of $\d_\es$ and $\ui\neq\d_\es$.

We now describe two variations of the ordinary contact process
on which this paper focuses.  Both processes have
three states, meaning that the processes take values 
in $\{-1,0,1\}^{\ZZ^d}$.

\subsubsection{Model A}
The  first process we consider has parameters
 $\ki,\kii,\li,\lii$ and $\hi,\hii$.  The state of a site
may change spontaneously from $1$ to $0$, from $0$ to $-1$,
from $-1$ to $0$ or from $0$ to $1$, at rates 
$\ki,\kii,\hii,\hi$ respectively.  Alternatively, a site which
is in state $-1$ or $0$ may change to state $0$ or $1$, 
respectively, at a rate proportional to the number of nearest
neighbours which are in state 1, the constants of proportionality
being given by $\lii$ and $\li$, respectively.  These transition rates
are informally summarized in the following table:
\begin{center}
\begin{tabular}{l|l}
Spontaneous rates & Neighbour rates\\\hline
$1\rightarrow 0$ rate $\ki$ & $0\rightarrow 1$ 
                                rate $\li\cdot$\#(type 1 nbrs) \\
$0\rightarrow -1$ rate $\kii$  & $-1\rightarrow 0$ 
                               rate $\lii\cdot$\#(type 1 nbrs) \\
$0\rightarrow 1$ rate $\hi$ & \\
$-1\rightarrow 0$ rate $\hii$ & \\
\end{tabular}
\end{center}
If $\kii=\lii=\hii=h=0$ we thus essentially recover the
ordinary 2-state contact process.
If $\kii=\lii=\hii=0$ but $h>0$ we obtain what may be called
the 2-state process \emph{with spontaneous infection}.

This 3-state process is closely related to a model proposed
to study the desertification of arid regions in \cite{KefiRietkAladoPueyo07}.
The intuition is that 0 represents a `vacant' patch of
`good' soil, 1 represents a vegetated patch, and $-1$
represents a `bad' patch of soil which must first be improved
(to state 0) before vegetation can grow there.  Type 1 patches
can influence the states of neighbouring patches either
by spreading seeds ($0\rightarrow1$) or improving the
soil ($-1\rightarrow0$), for example by binding the soil
better with  roots.  
Much less is known about the this 3-state process than
about the ordinary contact process with two states.  
To a large extent this is 
because the notion of `path' along which
infection spreads is no longer sufficient.  
In particular, we do not know if there is a unique stationary
distribution if all the parameters are strictly positive, as is
the case for the 2-state process
with spontaneous infection.  However, most of our results
on the 3-state process are
conditional on the assumption that there is a unique
stationary distribution $\ui$ in this situation. 
A more precise formulation of the assumption is stated in 
Assumption~\ref{assump}.

\subsubsection{Model B}
The second process we study  is close
to a process studied by Remenik~\cite{Remen08}.
The parameters are $\ki,\kre,\li,\hi,\hii$, and the transitions
are summarized in the following table.
\begin{center}
\begin{tabular}{l|l}
Spontaneous rates & Neighbour rates\\\hline
$1\rightarrow0$ rate $\ki$ & 
$0\rightarrow 1$ rate $\li\cdot$\#(type 1 nbrs)\\
($0$ or $1$) $\rightarrow -1$ rate $\kre$ & \\
$0\rightarrow 1$ rate $\hi$ & \\
$-1\rightarrow 0$ rate $\hii$ & \\
\end{tabular}
\end{center}
Thus a site changes state to $-1$ at rate $\kre$
\emph{regardless of} the current state, and transitions out 
of the state $-1$ occur at a rate independent of the number 
of type 1 neighbours.  
In  light of this observation, it is possible to interpret Model B as (ordinary) \emph{contact process in a random environment}. 

In the case $\hi=0$ this process is the one studied in~\cite{Remen08}.  
Remenik puts it forward as a model for the spread of vegetation, with
a slightly different interpretation of transitions
to state $-1$ than in Model A.
In Remenik's model the interpretation is that
``if a site becomes uninhabitable, the particles living
there will soon die'' (quote from~\cite{Remen08}).
In Model A, however, transitions to $-1$ only occur
for uninhabited sites (in state $0$) with the motivation
that they ``may undergo further degradation, for example,
by processes such as erosion and soil-crust formation''
(quote from~\cite{KefiRietkAladoPueyo07}).

Model B is considerably easier to study than 
Model A.  
Indeed, for the case $\hi=0$,
Remenik interpreted the model as 
a hidden Markov chain and, building on results by 
Broman~\cite{Broma07}, proved strong results such as
complete convergence.  
A key tool to obtaining this result is a \emph{duality} relation, 
which fails for Model A. 
For the case $\hi>0$,
exponential convergence to a unique
 invariant distribution
is stated in Lemma~\ref{exp_lem} below.

\subsection{Density-driven contact processes}
It is straightforward to generalize the definitions of
the contact processes we consider to allow time-varying infection rates
$\l(t)$ and $h(t)$ (see Section~\ref{graphical_sec} for more on this).
Furthermore, in the context of vegetation spread it seems natural to 
allow the rates governing
transitions from state 0 to 1
($\li$ and $\hi$) 
to depend on the overall
density of 1's in the process itself.  For example, one may
imagine that seeds can be blown over large distances to spread
vegetation, and that whether a seed which has landed on
a vacant piece of soil indeed becomes a plant may depend
on the overall competition of the other plants.  Indeed, the model proposed 
in~\cite{KefiRietkAladoPueyo07}
includes such a mechanism.  The model there is defined
in discrete time and in a finite region, 
and it is not immediately obvious that
it is possible to define such a process 
in continuous time and on the infinite 
graph $\ZZ^d$.
However, in
Section~\ref{ddcp_sec} we prove the existence of the following
class of processes.

Let $X(t)$ be a translation invariant
3-state process (Model A or B), and write
$\rho(t)=P(X_0(t)=1)$ for the \emph{density} of the process.
\begin{definition}[\ddcp]\label{ddcp_def}
Let the functions
$\dfl,\dfh\colon[0,1]\rightarrow[0,\oo)$ be given,
and let $X(t)$ be 
a translation-invariant 3-state contact process with 
parameters $\ki$, $\kii$, $\lii$, $\hii$ and $\li(\cdot)$,
$\hi(\cdot)$ in Model A or 
with parameters $\ki$, $\kre$, $\hii$, and $\li(\cdot)$, $\hi(\cdot)$ in Model B. 
This process is called a
\emph{density-driven contact process specified by $\dfl$ and $\dfh$}
if $\li,\hi$ satisfy $\li(t)=\dfl(\rho(t))$
and $\hi(t)=\dfh(\rho(t))$ for all $t\geq0$. 
\end{definition}

We use the abbreviation {\ddcp} for `density-driven
contact process'.
Intuitively a
{\ddcp} constantly updates its infection rates based
on the current prevalence of 1's.

\subsection{Outline}
In Section~\ref{results_sec} we
state our main results, which concern on the one hand 
`sharpness' and on the other `lack of robustness'. 
In Section~\ref{robust_sec} we prove our results on lack of
robustness, deferring the proofs of our sharpness results to 
Section~\ref{sharp_sec}.
In Section~\ref{prel_sec} we describe methods and results from the
literature which are needed
for the proofs of our main sharpness results.  
Section~\ref{sharp_sec} contains the proofs of our sharpness result 
(Theorem \ref{3st_sharp_thm}).
In Section~\ref{ddcp_sec}, we 
prove in general the existence of density-driven processes.

\section{Main results}
\label{results_sec}

We first formulate a condition about exponentially fast convergence to
a unique equilibrium measure for Model A. For Model B, this assumption
can be established using standard techniques.  Subsequently, we
formulate our results on lack of robustness (in Section
\ref{sec-robust}) and sharpness (in Section \ref{perc_sec}).

\subsection{Convergence to equilibrium}
Consider Models A and B
with constant parameters.
Henceforth, we assume that 
all parameters are positive, so we assume 
 $\ki,\kii,\li,\lii,\hi,\hii>0$ for Model A
and $\ki,\kre,\li,\hi,\hii>0$ for Model B.

It is well-known (and can be easily proved by a standard 
coupling argument using 
the graphical representation, see 
Section~\ref{graphical_sec}) that the assumption $h>0$ implies
exponentially fast convergence to equilibrium in Model B. 
For any $\xi\in\{-1,0,1\}^{\ZZ^d}$ 
let us write $\mu^\xi_t$ for the law of the contact process
with initial state $\xi$. Further, for a finite set
$\L\se\ZZ^d$ let $\mu^\xi_{t;\L}$ denote the restriction
of $\mu^\xi_t$ to $\L$.  Similarly, let $\ui_\L$
denote the restriction of the upper invariant measure $\ui$ to $\L$
(i.e., marginal of $\mu^\xi_t$ on 
$\{-1,0,1\}^{\L}$).
\begin{lemma}\label{exp_lem}
For Model B with $h>0$ and any
initial state $\xi$ we have that
\[
d_{\mathrm{tv}}(\mu^\xi_{t;\L},\ui_\L)\leq
|\L|e^{-ht}.
\]
\end{lemma}

For Model A we have not been able 
to establish exponential convergence to
equilibrium along the lines of Lemma \ref{exp_lem}.
However, it is natural to suppose that such a 
result should hold when all parameters
$\ki,\kii,\li,\lii,\hi,\hii>0$.  
Our results for Model A rely on such a convergence 
result, which we now formulate: 
\begin{assumption}[Exponential convergence to equilibrium]\label{assump}
For Model A with 
strictly positive parameters 
\begin{enumerate}
\item there is a unique stationary
distribution $\ui$, and
\item there are constants $C_1,C_2>0$ such that for 
all finite $\L\se\ZZ^d$ and all initial configurations
$\xi$ we have 
\[
d_{\mathrm{tv}}(\mu^\xi_{t;\L},\ui_\L)\leq
C_1|\L|e^{-C_2t}.
\]
\end{enumerate}
\end{assumption}
The particular places where Assumption~\ref{assump} is needed are 
in Lemma~\ref{rsw_lem} and in the
proof of Theorem~\ref{3st_sharp_thm} at the 
end of Section~\ref{sharp_pf_sec}. 

Here are some heuristic arguments supporting the assumption. 
First, as already noted  the assumption holds straightforwardly 
in Model B, and it also holds for the two-state contact process 
\emph{with spontaneous infections} 
which may be obtained from Model A by setting 
$\kii=\lii=\hii=0$.
Compared to the 2-state process, the extra state $-1$ in
the 3-state process introduces `delays' during which particles are
insensitive to infection attempts. 
The delay periods are of random length but with exponential tails,
and hence we do not expect the qualitative properties of convergence speed to
equilibrium to be different from the 2-state case.  
Also, by standard general arguments 
(see~\cite[Theorem~4.1]{liggett85}),
Assumption~\ref{assump} holds 
for a certain parameter range, namely when the 
`spontaneous rates'  are sufficiently large
compared with the `neighbour rates'.
Finally, if we couple the system starting with all sites having value
$-1$ with the system starting with all sites having value $1$, it
seems, intuitively, that the rules of the coupled dynamics provide a
stronger tendency to eliminate existing disagreement sites (i.e. sites
where the two systems differ in value) than to create new disagreement
sites. However, we have not been able to turn this into a proof.

\subsection{Percolation and the question of robustness}
\label{sec-robust}

For any configuration $\eta\in\{-1,0,1\}^{\ZZ^d}$,
consider the subgraph induced by sites in state 1 and the nearest-neighbor relation. 
Let $C_0$ denote the connected component of the subgraph of 1's containing the origin $0$, 
and write $|C_0|$ for the number of sites in $C_0$. 
If $\nu$ is a probability measure on 
$\{-1,0,1\}^{\ZZ^d}$, we say that \emph{percolation occurs}
under $\nu$ if $\nu(|C_0|=\infty)>0$.
For the rest of this section we fix $d=2$.

A major focus of this article is to study the phenomenon
of percolation when $\nu$ is an invariant measure of 
a 3-state contact process, 
possibly density-driven.
Indeed,
one of the main motivations is an intriguing suggestion 
in~\cite{KefiRietkAladoPueyo07} concerning a specific version 
of the density-driven Model A (with explicitly given forms of the
functions $\hi(t)$ and $\li(t)$, involving certain parameters). 
In our context (where the medium
is the \emph{infinite} lattice and time is continuous) 
that version is given by the functions 
in~\eqref{kefi-f_2-eq} below.  The suggestion 
in~\cite{KefiRietkAladoPueyo07} is that this model has a form of 
`robust critical behaviour': that there is a non-negligible
set of parameter values  for which the model has an invariant measure 
under which the size of an
occupied cluster has a power-law distribution. 

As the authors of~\cite{KefiRietkAladoPueyo07} remark, such behaviour
is different ``from classical critical systems, 
where power laws only occur at the transition point''.
Further, the authors suggest that this uncommon behaviour may be explained
by strong local positive interactions. 
(The latter means that the transitions from $-1$ to $0$ and the transitions
from $0$ to $1$ are `enhanced' by the presence of occupied 
sites in the neighborhood). Later in their paper they argue
that an important aspect to explain their `observed' robust critical 
behaviour would be that `disturbances' (transitions
to the $-1$ state) do not affect directly the occupied sites: 
they first have to change to the $0$ state, which
``constrains the spatial propagation of the disturbance''. 
In later life-sciences papers the robust criticality is 
debated~\cite{anti-kefi-comment,kefi11,anti-kefi,anti-kefi-reply}. 

The arguments in~\cite{KefiRietkAladoPueyo07} and those in the articles
mentioned above
lack mathematical rigour.
Our aim is to contribute by lifting the discussion to a 
rigorous mathematical level, and by proving
mathematical theorems that are 
relevant for the above mentioned discussion.
Our following result, Theorem~\ref{nonrobust_thm_3},  
shows (under Assumption~\ref{assump}) that in our formulation of the model 
in~\cite{KefiRietkAladoPueyo07}, criticality
is rare, in a strong and well-defined sense.
We also show a more general though weaker statement of a similar form
(Theorem~\ref{nonrobust_thm}).

\begin{definition}\label{critical_def}
We call a distribution $\nu$ on $\{-1,0,1\}^{{\mathbb Z}^2}$ 
\emph{critical} (for percolation)  if  
$\nu(|C_0|\geq n)$ converges to zero
subexponentially;    that is, $\nu(|C_0|\geq n)\rightarrow0$
as $n\rightarrow\oo$, but
\[
\liminf_{n\rightarrow\oo}\frac{-\log \nu(|C_0|\geq n)}{n}=0.
\]
\end{definition}

Calling such a distribution $\nu$ `critical' may be 
somewhat imprecise, partly as it seems to ignore the
possibility of a discontinuous phase transition.
However, the name is meant to capture the idea that
power law cluster sizes are associated with
critical behaviour.

The precise form of the {\ddcp} corresponding to the model 
in~\cite{KefiRietkAladoPueyo07}
is given by
\begin{equation} \label{kefi-f_2-eq}
\begin{split}
&\dfl(\rho) = \beta \, \frac{1 - \delta}{4} \, (\eps - g \rho), \\
&\dfh(\rho) = \beta \delta \, \rho \, (\eps - g \rho),
\end{split}
\end{equation}
where $\beta$, $\eps$ and $g$ are positive parameters
and $\delta\in(0,1)$.  This choice of functions is motivated
in the Methods supplement to~\cite{KefiRietkAladoPueyo07}.
Briefly, $\beta$ represents the seed production rate,
$\delta$ the fraction of seeds that are spread over
long distances, $\eps$ the establishment probability
of a seed not subject to competition, and $g$ a
competitive effect due to the presence 
of other plants.

For the {\ddcp} where $\lambda$ and $h$ 
are density-dependent and given by~\eqref{kefi-f_2-eq}
we have the following result:

\begin{theorem} [Lack of robustness]\label{nonrobust_thm_3}
Let $d=2$ and recall Definition~\ref{critical_def}.
\paragraph{$\blacktriangleright$ Model A}
Consider Model A under Assumption~\ref{assump}, 
with $\dfl(\cdot)$  and $\dfh(\cdot)$ given by~\eqref{kefi-f_2-eq}. 
Then for almost all 
$\ki,\kii,\lii,\hii$, $\b$, $\d$, $\eps$ and $g$, 
 the 3-state {\ddcp} does not have a critical invariant measure.
\paragraph{$\blacktriangleright$ Model B}
Similarly, consider Model B with $\dfl(\cdot)$  and $\dfh(\cdot)$ 
given by~\eqref{kefi-f_2-eq}. 
Then for almost all $\ki,\kre,\hii$, $\b$, $\d$, $\eps$ 
and $g$, the 3-state {\ddcp} does not have a critical invariant measure.
\end{theorem}

We also have the following result, which
on the one hand holds for much more general $\dfl,\dfh$,
but on the other hand has a  weaker conclusion.
We say that two functions $f,g\colon[0,1]\rightarrow\RR$ 
differ at 
most $\eps$ if $|f(r) - g(r)| < \eps$ for all $r\in[0,1]$.
The result is formulated and proved 
for Model A, but straightforwardly extends
to the Model B as well (with Assumption~\ref{assump}
replaced by Lemma~\ref{exp_lem}).

\begin{theorem} \label{nonrobust_thm}
Let $\dfl,\dfh$ be continuous, strictly positive
functions, and 
suppose the 3-state {\ddcp} with dynamics given by Model A and parameters
$\ki, \kii, \lii$, $\hii>0$ and $\dfl, \dfh>0$
has a critical invariant distribution.
Then, under Assumption~\ref{assump}, 
for every $\eps>0$ there are parameters
$\ki', \kii',\lii'$,  $\hii'$ and $\dfl', \dfh'$ 
which each differ at most $\eps$ 
from the original parameters, and
for which the corresponding {\ddcp} has no critical 
invariant measure.
\end{theorem}

Theorems~\ref{nonrobust_thm_3} and~\ref{nonrobust_thm} 
are proved in Section~\ref{robust_sec}.  These results
cast considerable doubt on the suggestions
in~\cite{KefiRietkAladoPueyo07} discussed in the beginning of this section.

\subsection{Sharpness of percolation transitions}
\label{perc_sec}

The main step in proving 
Theorems~\ref{nonrobust_thm_3} and~\ref{nonrobust_thm}  
is to establish 
\emph{sharpness results} for percolation under the invariant
measures of contact processes, which we state in this section.  
Such results are also of independent interest.  
Given these sharpness results, the proofs of
Theorems~\ref{nonrobust_thm_3} and~\ref{nonrobust_thm} are relatively
elementary. 
For $x,y\in\RR^k$ we use the notation $x\prec y$
to indicate that each coordinate of $x$ is strictly
smaller than the corresponding coordinate of $y$.

For Model A, we require Assumption~\ref{assump} (which, for Model B, has been verified in Lemma \ref{exp_lem}). 
By comparison with Bernoulli 
percolation it follows immediately that
$\ui(|C_0|=\oo)>0$ provided $\hi,\hii$ are large enough,
or $\hii>0$ and $\hi$ is large enough.
In Section~\ref{sharp_sec} we prove the following
result.
\begin{theorem}[Sharpness for 3-state contact process]\label{3st_sharp_thm}
Consider the 3-state contact process with $d=2$.
\paragraph{$\blacktriangleright$ Model A}
Fix $\ki,\kii>0$. 
Under Assumption~\ref{assump} we have the following.
If the parameters $\li,\lii,\hi,\hii>0$ are such that
$\ui(|C_0|=\infty)=0$, then whenever
$(\li',\lii',\hi',\hii')\prec (\li,\lii,\hi,\hii)$,
there is $c >0$
such that $\ui(|C_0|\geq n)\leq e^{-c n}$ for all $n\geq1$.
\paragraph{$\blacktriangleright$ Model B}
Fix $\ki,\kre>0$.
If the parameters $\li,\hi,\hii>0$ are such that
$\ui(|C_0|=\infty)=0$, then whenever
$(\li',\hi',\hii')\prec (\li,\hi,\hii)$,
there is $c >0$
such that $\ui(|C_0|\geq n)\leq e^{-c n}$
for all $n\geq1$.
\end{theorem}

Theorem~\ref{3st_sharp_thm} has the following consequence, which will be used in the
proof of Theorem \ref{nonrobust_thm_3} in Section 3.2..  
For fixed $\ki,\kii/\kre,\li,\lii,\hii>0$ define 
$h_\perc =h_\perc(\l) :=\inf\big\{h \geq0:\ui(|C_0|=\oo)>0\big\}.$

\begin{corollary}\label{3st_sharp_cor}
Consider the 3-state contact process with $d=2$.  
\paragraph{$\blacktriangleright$ Model A}
Under Assumption~\ref{assump}, for all $\ki,\kii>0$
and almost all $\lii,\hii>0$ the following holds:
for all but countably many $\li>0$, if $h<h_\perc(\li)$
then $\ui(|C_0|\geq n)\leq e^{-c n}$
for some $c>0$ and all $n\geq1$.
\paragraph{$\blacktriangleright$ Model B}
For all $\ki,\kre>0$
and almost all $\hii>0$ the following holds:
for all but countably many $\li>0$, if $h<h_\perc(\li)$
then $\ui(|C_0|\geq n)\leq e^{-c n}$
for some $c>0$ and all $n\geq1$.
\end{corollary}

\begin{remark}
The 2-state contact process with spontaneous infection can
be obtained from Model B by letting $\tilde h=\kre=0$ (the state $-1$ thus
plays no role).  Although our results are formulated under the
assumption that all parameters are positive, it may be seen quite
straightforwardly from the proofs that all results of Section
\ref{results_sec} apply also to this model.
\end{remark}

\begin{remark}\label{rem-sharp}
A straightforward modification of the proof of 
Corollary \ref{3st_sharp_cor} 
gives the following statement for Model A. 
Under Assumption~\ref{assump}, for 
all $\ki,\kii>0$, and almost all $\lii,\hii>0$, the 
following holds for all but countably many $h>0$: 
if $\lambda<\lambda_\perc(h):=\inf\big\{\lambda \geq0:\ui(|C_0|=\oo)>0\big\}$ then $\ui(|C_0|\geq n)\leq e^{-c n}$
for some $c>0$ and all $n\geq1$.
In other words, for almost all 
choices of the parameters $\lii, \hi, \hii, \ki, \kii$ 
there is at most one value of $\lambda$ for which
$\ui(|C_0|\geq n)$ goes to $0$ slower than exponentially.
We may deduce a similar statement for Model B. 
\end{remark}

\begin{remark}
Our arguments do not, however, allow us to go beyond the ``almost all''
in Corollary~\ref{3st_sharp_cor} 
(and Remark~\ref{rem-sharp}). That is, we are not able to prove
that the percolation transition in $h$ (respectively, $\lambda$) 
is sharp
for \emph{arbitrary} fixed values of the other parameters.
More generally, we have not proved the stronger version of 
Theorem~\ref{3st_sharp_thm}  where just one (instead of all) 
the `good' parameters are
 decreased.

The essential difficulty is, informally, the following. To obtain such
a stronger version of Theorem~\ref{3st_sharp_thm}, 
we need to suitably compare the
effect of a small change of one parameter with the effect of 
changes of other parameters. Some
comparisons are simple: it is easy to see 
that the system obtained by increasing $h$ by
$\varepsilon$, dominates the system obtained by increasing $\lambda$
by $\varepsilon/4$. However, it is not obvious how (e.g.) to compare
an increase in $h$ with an increase in $\tilde h$. Although there is a
general approach to such and related problems (see
e.g.~\cite{aiz-gr-strict} 
and Sections 3.2 and 3.3 in~\cite{Gr99})
the concrete applicability of that approach depends
very much on the details of the model. Moreover, as pointed out 
in~\cite{bal-boll-riord}, even in some `classical'
percolation models, the technical problems that arise are more
delicate than expected earlier.

Therefore, and because our current version of Theorem~\ref{3st_sharp_thm}  
is strong
enough to obtain Theorem~\ref{nonrobust_thm_3} 
(and the statement of this latter theorem
would not essentially benefit from the mentioned stronger version of
the former), we do not  pursue such improvements in this
paper.
\end{remark}
%%%%%%%%%%%%%%%%%%%%%%%%%%%%%%%%%%%%%%%%%%%%%%%%%%%%%%%

\section{Proofs of nonrobustness for density-driven processes}
\label{robust_sec}

In this section we prove Theorems~\ref{nonrobust_thm_3} 
and~\ref{nonrobust_thm}
on lack of robustness in the model
proposed in~\cite{KefiRietkAladoPueyo07} assuming the 
sharpness results of Corollary~\ref{3st_sharp_cor}. 
We begin with a discussion
about the stationary distributions of {\ddcp}.

\subsection{Stationary distributions for density-driven processes}
\label{stat_sec}
We first consider Model A. 
Let $\kappa, \kii, \lii, \hii > 0$ be fixed.
Let $X(t)$, $t \geq 0$, be a 
{\ddcp} for the parameters $\kappa$, $\kii$, $\lii$, $\hii$, $\dfl$ and $\dfh$. Suppose that
$X$ is stationary, i.e.\ 
the distribution of $X(t)$ is constant in $t$. 
Denote this distribution by $\nu$.
By stationarity, the occupation density $\rho(t)$ of $X(t)$ is
constant, say $\rho(t)\equiv\rho$.
Writing $\l = \dfl(\rho)$ and $h=\dfh(\rho)$ we thus 
find that $\nu$ is a stationary distribution 
for the contact process with \emph{constant} parameters
$\lambda$, $h$, $\kappa$,  $\kii$, $\lii$,  and $\hii$.

Suppose $\dfl(\rho),\dfh(\rho)>0$ for all $\rho\in[0,1]$.  
Since by Assumption~\ref{assump} there is only one stationary distribution $\ui$, it follows that $\nu=\ui$.
Let $\ol\rho(\l, h)=\ui(\{\eta : \eta_0=1\})$ 
denote the density of $\ui$.
It follows that
$\l$ and $h$ satisfy the fixed point equations 
$\l = \dfl(\ol\rho(\l,h))$ 
and $h=\dfh(\ol\rho(\l, h))$, 
respectively.  Conversely,
if $h$ and $\l$ satisfy these fixed point equations, 
then $\ui$ is stationary for the {\ddcp} defined by $\dfl, \dfh$.  
We summarize these findings in the
following proposition:
\begin{proposition}\label{3st_stat_prop}
Let $\ki$, $\kii$, $\lii$, $\hii>0$
be fixed.
Suppose $\dfl(\rho),\dfh(\rho)>0$ for all $\rho\in[0,1]$.  
Then, under Assumption~\ref{assump}, the stationary
distributions of the 3-state {\ddcp} specified by $\dfl,\dfh$ 
are precisely the measures $\ui$ for 
$\l$, $h$ satisfying
$\l = \dfl(\ol\rho(\l, h))$ and $h=\dfh(\ol\rho(\l, h))$.
\end{proposition}
The corresponding result is valid for Model B.

\subsection{Proof of Theorem~\ref{nonrobust_thm_3}}
We prove the theorem using Corollary~\ref{3st_sharp_cor}. 
Writing $\g_1 = \b (1 - \d)/4$ and $\g_2 = \b \d$, it is sufficient to
prove the following claim: 
\emph{For almost all $\k$, 
$\kii/\kre$, $\hii$,
$\g_1$, $\g_2$, $\eps$ and $g$, 
the {\ddcp} with
$\dfl(\rho) = \g_1(\eps - g \rho)$ and 
$\dfh(\rho) = \g_2 \rho (\eps - g \rho)$ does not have a
critical invariant measure.}   In the argument that follows we will 
frequently use the fact that the product of a measure zero set with
any measurable set has measure zero.

We give a proof for Model A only, the proof for Model B is similar. 
For fixed $\ki,\kii,\lii,\hii$ we 
call $\li$ \emph{bad} if  $\ui(|C_0|\geq n)$
does not decay exponentially for all $h<h_\perc(\li)$.
For all $\ki,\kii$ and almost all $\lii,\hii$, 
Corollary~\ref{3st_sharp_cor} implies that the set
of bad $\li>0$ is at most countable.  
We henceforth assume
that $\ki,\kii,\lii,\hii$ are fixed and chosen so that
the set of bad $\li$ is at most countable.

Suppose the {\ddcp} has a critical 
invariant measure $\nu$.  
Note that, since $\nu$ is invariant,  
$\rho, \li, h$ do not vary with $t$, and
that, since $\nu$ is critical, $\rho$, and hence 
$\li$ and $h$, are $> 0$. 
We now consider the two cases, $\li$ `bad' or not.
Here is a brief summary of the argument that follows.
In the case when $\l$ is not bad we use the fact that $h$, and
hence also $\rho$, is then a function of $\l$ only.  These additional
constraints, together with~\eqref{kefi-f_2-eq}, allow us (roughly
speaking) to write the remaining parameters in terms of $\l$ and then
to deduce the result from the precise form of this expression
(see~\eqref{rob-eq-g2-2}).  In the case when $\l$ is bad it suffices
to show that the set of possible choices of the remaining parameters
has measure zero, since there are only countably many such $\l$.  This
is done by using~\eqref{kefi-f_2-eq} and 
Proposition~\ref{3st_stat_prop} to write additional relations among
these parameters.

We now turn to the argument proper.
If $\li$ is \emph{not} bad 
then, since $\nu$ is critical, $\hi$ must equal $\hi_\perc(\li)$. 
Hence the following two equations hold, where
$\rho_\perc(\l) = \rho(\l, h_\perc(\l))$ denotes the density of the upper
invariant measure with parameters $\l$ and $h = h_\perc(\l)$:

\begin{equation} \label{kef_fp_1_eq}
\l = \g_1 (\eps - g \rho_\perc(\l)),
\end{equation}
\begin{equation} \label{kef_fp_2_eq}
h = h_\perc(\l) = \g_2 \rho_\perc(\l) (\eps - g \rho_\perc(\l)).
\end{equation}

To prove the claim, fix also $g$ and $\g_1$. 
With these parameters fixed, it is clear that
for each $\l$ there is at most one $\eps$ such that \eqref{kef_fp_1_eq} holds. 
Hence the measure of the set of pairs $(\l,\eps)$ such that
\eqref{kef_fp_1_eq} holds is zero.  It follows from Fubini's theorem
that  for almost all $\eps$ the set 
$L = L(\eps) = L(\eps, \k, g,\g_1)$ 
of those $\l$ for
which \eqref{kef_fp_1_eq} holds has Lebesgue measure $0$. 
(Note that $\rho_\perc(\l)$ is measurable since $\rho(\l,h)$ and
$h_\perc(\l)$ are measurable.)
Now also fix (besides the above mentioned parameters
which were already fixed) the parameter $\eps$ 
such that $L$ indeed has measure $0$.
Note that for each $\l$ there is at most one $\g_2$ 
such that \eqref{kef_fp_2_eq} holds. Let $L' \subset L$
be the set of those $\l \in L$ for which there is indeed such a
$\g_2$. Rearranging \eqref{kef_fp_2_eq} we can
write this $\g_2$ as a function of $\l \in L'$:

$$\g_2 = \frac{h_\perc(\l)}{\rho_\perc(\l) (\eps - g \rho_\perc(\l))}.$$
Using \eqref{kef_fp_1_eq}, we can 
`eliminate' $\rho_\perc(\l)$ from the above expression for $\g_2$ and get

\begin{equation} \label{rob-eq-g2}
\g_2 = \frac{h_\perc(\l) g \g_1^2}{\l (\eps \g_1 - \l)}, \,\,\,\, \l \in L'.
\end{equation}
Write the right hand side of \eqref{rob-eq-g2} as a function 
\begin{equation}\label{rob-eq-g2-2}
F(\l) :=  \frac{h_\perc(\l) g \g_1^2}{\l (\eps \g_1 - \l)}.
\end{equation}

We want to show that $F(L')$ has measure $0$. To do this, 
we note that $h_\perc(\l)$ is 
uniformly Lipschitz continuous in $\l$:
for each $\alpha \geq 0$,
the  process with parameters 
$\l + \alpha$ and $h$ is stochastically dominated by the 
process with parameters $\l$ and $h + 4 \alpha$
(this is intuitively obvious from the 
description of the dynamics, and can be easily proved
using the graphical representation of
Section~\ref{graphical_sec}), so
\begin{equation} \label{eq-hp-cont}
h_\perc(\l) \geq h_\perc(\l + \alpha) \geq h_\perc(\l) - 4 \alpha.
\end{equation}
Hence the numerator in the definition of the function $F$ above 
is locally Lipschitz. It follows that $F$ is
locally Lipschitz outside the point $\l = \eps \g_1$. 
Hence, since $L'$ has measure $0$, $F(L')$ also
has measure $0$ 
(any cover of $L'$ with `small' intervals
is mapped under $F$ to a cover of $F(L')$
with comparably small intervals).

We next consider the case when $\li$ is bad.
For a fixed bad $\li$, we can use Proposition~\ref{3st_stat_prop}
to write  $\rho=\rho(\hi)$.
We can no longer conclude that $\hi=\hi_\perc(\li)$,
but we still have the relations
\begin{equation} \label{kef_fp_3_eq}
\li = \g_1 (\eps - g \rho(\hi)),\quad\mbox{and}\quad
\hi = \g_2 \rho(\hi) (\eps - g \rho(\hi)).
\end{equation}
We aim to show that, for each fixed bad $\li$,
the set of choices of the parameters
$\g_1,\g_2,\eps,g$ such 
that~\eqref{kef_fp_3_eq}
holds has measure zero.  This concludes the proof
since a countable union of measure zero sets
has measure zero.

If~\eqref{kef_fp_3_eq}
holds, we may rearrange to obtain the relations
\begin{equation} \label{kef_fp_5_eq}
\rho(\hi) = \frac{\g_1}{\li\g_2} h,\quad\mbox{and}\quad
\eps=\frac{\li}{\g_1}+g\rho(\hi)=
\frac{\li}{\g_1}+g\frac{\g_1}{\li\g_2} h.
\end{equation}
We now fix arbitrary $\g_2,g>0$.  
Then for almost all $\g_1$, the
first relation in~\eqref{kef_fp_5_eq}
can only hold for $\hi$ in a set of measure zero.
This follows from Fubini's theorem (using 
polar coordinates) and the fact that
the set $\{(\hi,\rho(\hi)):h>0\}$
has two-dimensional Lebesgue measure zero.
We now fix $\g_1$ such that the
first relation in~\eqref{kef_fp_5_eq}
only holds for $\hi$ in a set of measure zero.
It follows that the second relation in~\eqref{kef_fp_5_eq}
can only hold for $\eps$ in a set of measure zero.
Using Fubini's theorem again 
this concludes the proof.
\qed

\subsection{Proof of Theorem \ref{nonrobust_thm}} 
Consider Model A with parameters 
$\ki$, $\kii$, $\li$, $\lii$, $\hi$, $\hii$,
all strictly positive.
Under Assumption~\ref{assump}, there is a unique
invariant distribution $\ui$ for this process, 
and we have (as stated precisely in the same Assumption)
exponentially fast convergence to that distribution, from any starting distribution.
Recall that $\ol\rho$ denotes
the density of $\ui$ (i.e.\! the probability
under $\ui$ that a given vertex has value $1$).
\begin{lemma} \label{Rlem-rc}
$\ol\rho$ is continuous from the right in each of the parameters 
$\li$, $\lii$, $\hi$ and $\hii$.
\end{lemma}
\begin{proof}
Let $\mu_t$ denote the distribution at time $t$ if we
start the process with all vertices in state $1$. 
By uniqueness,  $\ui$ is the limit, as $t \rightarrow \infty$, of 
$\mu_t$.
From Lemma~\ref{3st_mon_lem} we have
that $\mu_t$ is stochastically increasing in each of the 
parameters $\hi$, $\hii$, $\li$ and $\lii$ (and stochastically
decreasing in $\ki$ and $\kii$). Also by obvious monotonicity, 
$\mu_t$ is stochastically decreasing in $t$.
For each $t\geq0$ the density $\rho(t)$ under 
$\mu_t$ is continuous in each of the
parameters $\li$, $\lii$, $\hi$, $\hii$,
and by the above we have that 
$\rho(t)\searrow\ol\rho$.
The result follows since we can interchange the order of any
two decreasing limits.
\end{proof}

\begin{proof}[Proof of Theorem \ref{nonrobust_thm}]
Let $\nu$ denote the critical invariant measure mentioned in the statement of the theorem.
Let $\rho$ denote its density, and let $\li = \dfl(\rho)$ and $\hi =
\dfh(\rho)$.   Then, as in Proposition~\ref{3st_stat_prop}, 
$\nu$ is invariant under the 3-state contact process dynamics with
parameters $\ki$, $\kii$, $\li$, $\lii$, $\hi$, $\hii$.
Hence, under Assumption \ref{assump}, $\nu$ is the unique 
measure $\ui$ for these parameters.
So we have 
\[
\begin{split}
\li &= \dfl(\rho(\ki, \kii, \li, \lii,\hi, \hii)),\mbox{ and}\\
\hi &= \dfh(\rho(\ki, \kii, \li, \lii,\hi, \hii)).
\end{split}
\]
Now we increase each of the `good' parameters 
$\li, \lii, \hi$ and $\hii$ by an amount $\in(0, \eps/2)$ so small that
\[
\dfl(\rho(\ki, \kii, \li, \lii,\hi, \hii)) \mbox{ and }
\dfh(\rho(\ki, \kii, \li, \lii,\hi, \hii))
\] 
change by at most $\eps/2$.
This is possible by the continuity of $\dfl$ and 
$\dfh$ and Lemma~\ref{Rlem-rc}.
Denote the new parameters by 
$\ki' = \ki$, $\kii' = \kii$, $\li',\lii',\hi',\hii'$.
From the above it follows that there are 
continuous functions 
$\dfl'$ and $\dfh'$ which differ at most
$\eps$ from $\dfl$ and $\dfh$ respectively, such that
\[
\begin{split}
\li' &= \dfl'(\rho(\ki', \kii', \li', \lii',\hi', \hii')) ,\mbox{ and}\\
\hi' &= \dfh'(\rho(\ki', \kii', \li', \lii',\hi', \hii')) 
\end{split}
\]
(For example, one may take 
$\dfl'(r) = \dfl(r) + \li' -\dfl(\rho(\ki',\kii', \li', \lii',\hi', \hii'))$, 
and take $\dfh'$ in a similar way).
Let $\nu'$ be the invariant measure for the contact process with fixed parameters
$\ki', \kii', \li', \lii',\hi', \hii'$.
From the above we conclude that $\nu'$ is invariant under the
{\ddcp} dynamics with
parameters $\ki', \kii', \lii', \hii'$,
 $\dfl'$ and $\dfh'$, and each of 
these `new' parameters differs at most $\eps$
from the corresponding `old' one. Moreover, by 
Theorem~\ref{3st_sharp_thm}, 
this $\nu'$ is not critical.
This completes the proof.
\end{proof}

%%%%%%%%%%%%%%%%%%%%%%%%%%%%%%%%%%%%%%%%%%%%%%%%%%%%%%%%%%%%%%

\section{Ingredients from the literature}
\label{prel_sec}

In this section we discuss a number of results and methods needed for 
the proofs of our main sharpness result, 
Theorem~\ref{3st_sharp_thm}.
First we discuss graphical representations for contact processes,
then some methods from percolation theory as well as influence results.

\subsection{Graphical representations}
\label{graphical_sec}
Central to the study of the contact process is a graphical
representation in terms of Poisson processes of `marks'
and `arrows'.  
For Model A this is as follows.   We write $D_1$ and $D_2$ for 
independent Poisson processes on $\ZZ^d\times[0,\oo)$ of intensities
$\ki$ and $\kii$, which we think of as the processes
of `down' marks ($1\rightarrow0$ and $0\rightarrow-1$, respectively).
Independently of these and of each other we define
Poisson processes $U_1$ and $U_2$ of `up' marks 
($0\rightarrow1$ and $-1\rightarrow0$)
of respective intensities $\hi$ and $\hii$.  Finally,
independently of all these and of each other
we define Poisson processes $A_1$ and $A_2$ 
of arrows ($0\rightarrow1$ and $-1\rightarrow0$, respectively)
with respective intensities $\li$ and $\lii$ on the ordered nearest-neighbor sites $xy$ times $[0,\oo)$.  The rates
of these processes are summarized in the following table:
\begin{center}
\begin{tabular}{l|l}
Spontaneous transitions & Neighbour transitions\\
(on $\{x\}\times[0,\oo)$) & 
(on $\{xy\}\times[0,\oo)$)\\\hline
$D_1$ rate $\ki$ & $A_1$ rate $\li$ \\
$D_2$ rate $\kii$ & $A_2$ rate $\lii$ \\
$U_1$ rate $\hi$ & \\
$U_2$ rate $\hii$ & \\
\end{tabular}
\end{center}
The interpretation of the Poisson processes is as usual with
interacting particle systems: At a point $(x,t)\in U_1$ the site $x$
changes from 0 to 1 (it does not change if it was in state $-1$ or 1
before), if $(xy,t)\in A_1$ and the process is in state 1 at $x$ and
in state $0$ at $y$, then $y$ changes to 1, etc.
It will later (in Section~\ref{sharp_sec}) be useful to
focus on the process of \emph{incoming} arrows on each line
$\{y\}\times[0,\infty)$, that is the collection of arrows
at points $(xy,t)$ for $x$ a neighbour of $y$.
For all $y$ the incoming arrows form  Poisson processes,
of intensities $2d\li$ and $2d\lii$ for $A_1$
and $A_2$ respectively.

We also consider the three state process with time-varying parameters
$\lambda$ and $h$. Such process is easily defined via its graphical
representation.  Let $\lambda(\cdot)$ and $h(\cdot)$ be nonnegative
integrable functions, and let $A_1$ and $U_1$ be independent Poisson
processes of rates $\lambda(\cdot)$ and $h(\cdot)$, respectively.

The process has the following 
monotonicity in the initial condition and in the graphical
representation.  Let $X$ denote the 3-state 
process with initial state
$\xi\in\{-1,0,1\}^{\ZZ^d}$ and graphical representation $D_1$, $D_2$,
$U_1$, $U_2$, $A_1$ and $A_2$, and let $X'$ denote the process 
with initial condition
$\xi'\in\{-1,0,1\}^{\ZZ^d}$ and graphical representation $D_1'$, $D_2'$,
$U_1'$, $U_2'$, $A_1'$ and $A_2'$.  If the following hold, then
$X'(t)\geq X(t)$ for all $t\geq0$:  $\xi'\geq\xi$,
$D_1'\se D_1$, $D_2'\se D_2$,
$U_1\se U_1'$, $U_2\se U_2'$, $A_1\se A_1'$ and $A_2\se A_2'$.

An analogous construction exists for Model B. 
The only changes are that the Poisson process $D_2$ has intensity 
$\kre$, and represents transition to state $-1$ 
irrespective of the previous state. 
Further, there is no process $A_2$ for the Model B. 

The monotonicity statement for these processes reads as follows.  
\begin{lemma}[Monotonicity]\label{3st_mon_lem}\label{re_mon_lem}
Model A is (stochastically)
increasing in the initial state and the parameters
$\li,\lii,\hi,\hii$ and decreasing in $\ki,\kii$.
Model B is (stochastically)
increasing in the initial state and the parameters
$\li,\hi,\hii$ and decreasing in $\ki,\kre$.
\end{lemma}

This monotonicity property implies for both processes that 
if the initial state $\xi$ consists of only 1's then the 
distribution of the process at time $t$ is stochastically 
decreasing in $t$.  Standard arguments then imply that
the process decreases (stochastically) to a limiting
distribution, which in both cases we denote by $\ui$
and call the \emph{upper invariant measure}.

\subsection{Finite-size criterion}

Next we present  a so-called finite-size criterion
for percolation. Its analog for Bernoulli percolation 
is a well-known classical result which,
as pointed out in
\cite{Berg11} (see Lemma 2.3 in that paper), can 
be generalized to the case where the configurations come from the
supercritical ordinary contact process. 
The same arguments as in \cite{Berg11} yield our Lemma \ref{fin_lem} below.

Let $d=2$ and 
let $H(m,n)$ denote the event that there is a left-right crossing of
the rectangle $[0,m]\times[0,n]$
(ie, that the subgraph of $[0,m]\times[0,n]$ spanned 
by sites in state $1$ contains a path from some
$(0,x)$ to some $(m,y)$ where $0\leq x,y\leq n$).
Let $V(m,n)$ denote the event that there is an up-down crossing of
the rectangle $[0,m]\times[0,n]$.

\begin{lemma}\label{fin_lem}
(For Model A we assume Assumption \ref{assump} here). \\
There is a (universal) constant $\hat \eps>0$ such that the following holds
for Model A under Assumption~\ref{assump} and for Model B.\\
 For
all strictly positive values of the parameters, there is $\hat n$
(depending on the parameters) such that 
\begin{enumerate}
\item If for some $n\geq\hat n$ we have 
$\ui(V(3n,n))<\hat\eps$, then there is $c>0$ such that 
$\ui(|C_0|\geq k)\leq e^{-ck}$ for all $k\geq0$.
\item If for some $n\geq\hat n$ we have 
$\ui(H(3n,n))>1-\hat\eps$, then 
$\ui(|C_0|=\oo)>0$.
\end{enumerate}
\end{lemma}

\subsection{An influence result}

We further need the following combination of the Margulis-Russo formula and
an influence result which essentially comes 
from Talagrand's paper \cite{Talag94}
(which in turn is closely related to~\cite{kkl}),
where all
the $p_i$'s in the description below are equal. 
For our particular situation we straightforwardly generalized
the form (with two different $p_i$'s) in 
\cite[Corollary~2.9]{Berg11}. 
See also e.g.~\cite{FriedKalai96} 
and~\cite{BolloRiord06,BolloRiord08}.

Let $X=(X_{i,j}:1\leq i\leq m, 1\leq j\leq n)$ be a collection of
independent $\{0,1\}$-valued random variables such that 
\[
P(X_{i,j}=1)=p_i\quad\mbox{for all } j\in\{1,\dotsc,n\}.
\]
For fixed $i,j$, let $X^{(i,j)}$ denote the random
vector obtained from $X$ by replacing $X_{i,j}$
with $1-X_{i,j}$ (and keeping all other $X_{i',j'}$
the same).  For an event $A$, define the \emph{influence}
of $X_{i,j}$ on $A$ as
\[
I_{i,j}(A)=P(\{X\in A\}\vartriangle\{X^{(i,j)}\in A\}),
\]
where $\vartriangle$ denotes symmetric difference.

\begin{lemma}\label{infl_lem}
Fix $k\in\{1,\dotsc,m\}$ and suppose that $H$ is an event
which is increasing in the $X_{i,j}$ for $i\leq k$,
and decreasing in the $X_{i,j}$ for $i\geq k+1$.
Let $N$ denote the number of indices $(i,j)$
such that $I_{i,j}(H)$ is maximal.  There is an
absolute constant $K$ such that 
\[
\sum_{i=1}^k\frac{\partial}{\partial p_i}P(H)-
\sum_{i=k+1}^m\frac{\partial}{\partial p_i}P(H)\geq
\frac{P(H)(1-P(H))}{K\max_i p_i \log(2/\min_i p_i)}\log N.
\]
\end{lemma}
For our application, $m$ represents the number of different types of symbols. 
We apply it with $m=6$ and $k=4$ for Model A, and
with $m=5$ and $k=3$ for the Model B.

\subsection{RSW-result}

The following result is usually referred to as a 
\textsc{rsw}-type result as this type of result was pioneered, for Bernoulli
percolation, in
papers by Russo, Seymour and Welsh.  
A highly non-trivial extension of (a weak version of) 
the original \textsc{rsw}-result to a dependent percolation model,
namely  the random
Voronoi model, was proved by Bollob\'as and Riordan
\cite{BolloRiord06} 
(and modified in~\cite{BeBrVa08} to a form which is closer to 
Lemma~\ref{rsw_lem} below).
As pointed out in \cite{BolloRiord06},
the result holds under quite mild geometric, positive-association 
and spatial mixing conditions.
In \cite{Berg11} (see Proposition 2.4 in that paper) it is explained 
that these conditions hold
for the supercritical ordinary contact process. 
The same arguments hold for our models.

\begin{lemma}\label{rsw_lem}
Consider the upper invariant measure $\ui$ for Model A under Assumption~\ref{assump} 
or Model B with $h>0$. 
If for some $\a>0$ we have 
$\limsup_{n\rightarrow\oo} \ui(H(\a n, n))>0$
then for all $\a>0$ we have 
$\limsup_{n\rightarrow\oo} \ui(H(\a n, n))>0$.
\end{lemma}

%%%%%%%%%%%%%%%%%%%%%%%%%%%%%%%%%%%%%%%%%%%%%%%%%%%%%%%%%%

\section{Proofs of sharpness results}
\label{sharp_sec}

In this section we prove Theorem~\ref{3st_sharp_thm} and Corollary~\ref{3st_sharp_cor}.
Here is an outline of the argument that follows.  
Suppose $\ui$ is an invariant measure
for which the cluster
size $|C_0|$ does \emph{not} have exponential tails.  The first 
part of Lemma~\ref{fin_lem} together with Lemma~\ref{rsw_lem}
imply that certain crossing probabilities then have
uniformly positive probability under $\ui$.  We want to 
apply Lemma~\ref{infl_lem} to show that,
with an arbitrarily small increase of the relevant parameters,
we can `boost' this to get crossing probabilities close
to 1.  The second part of Lemma~\ref{fin_lem} then tells
us that $|C_0|$ is now infinite with positive
probability.

One of the main technical obstacles with carrying out this argument
is that Lemma~\ref{infl_lem} applies to events
which are defined in terms of a finite number of Bernoulli
variables, whereas contact processes are defined in terms
of `continuous' objects (Poisson processes).  The first
step is therefore a \emph{stability coupling}, a type of coupling which was also used
in \cite{BolloRiord06} for the Voronoi model, 
and later in \cite{Berg11} for the ordinary contact process. 
It tells us that
if we increase the `good' parameters then we can 
encode the contact process sufficiently well 
in terms of Bernoulli variables.  This is the topic of 
Section~\ref{stab_sec}. We give a detailed proof of
this part of the argument for Model A because it is
considerably more complicated than the corresponding one 
in~\cite[Lemma~3.2]{Berg11} for the
ordinary contact process. 
(For Model B we give an outline.)
The subsequent parts of the proof appear in Section~\ref{sharp_pf_sec}.
Recall that we are only considering the planar case
$d=2$.  

We start by pointing out that the monotonicity lemma 
(Lemma~\ref{3st_mon_lem})
implies that Theorem~\ref{3st_sharp_thm} follows once we establish the following claim: 
Let $\ki,\kii,\li,\lii,\hi,\hii>0$ be fixed, and 
consider the parameterization
$\{(\ki,\kii,r \li,r \lii,r \hi,r \hii):r\geq0\}$ for Model A;
and let $\ki,\kre,\li,\hi,\hii>0$ be fixed, and 
consider the parameterization
$\{(\ki,\kre,r \li,r \hi,r \hii):r\geq0\}$ for Model B.
Define 
$r_\perc=\inf\{r\geq0: \ui(|C_0|=\oo)>0\}$
as a function of the other parameters. 
Then the percolation transition is sharp in $r$, in
that if $r<r_\perc$ then $\ui(|C_0|\geq n)\leq e^{-cn}$
for some $c>0$.

It is convenient to
rescale time so that the total rate of `events per line'
is 1.  That is, we assume 
\begin{equation}\label{time_eq}
\begin{split}
&\ki+\kii+4\li+4\lii+\hi+\hii=1\qquad\mbox{(Model A)},\\
&\ki+\kre+4\li+\hi+\hii=1\qquad\mbox{(Model B)}.
\end{split}
\end{equation}
This clearly leaves the invariant measure $\ui$ unchanged.
Write 
\begin{equation*}
q=\left\{
\begin{array}{ll}
4\li+4\lii+\hi+\hii
&\mbox{(Model A)},\\
4\li+\hi+\hii
&\mbox{(Model B)}.
\end{array}\right.
\end{equation*}
Thus $q$ equals the part of the sum in~\eqref{time_eq}
which is decreased in  the sharpness theorems. 
We vary the parameter $q\in[0,1]$
while keeping the sum~\eqref{time_eq}
constant. 
We are then required to prove, under the relevant
assumptions, that if $q$ is such that 
$\ui(|C_0| \geq n)$ goes to $0$ as $n \rightarrow \infty$, 
but slower than exponentially, then for any $q''>q$,
\[
\ui(|C_0|=\oo) > 0.
\]

The objective of the stability coupling is the following.
We wish to discretize time into intervals of length 
$\delta=n^{-\alpha}$ (for a certain $\alpha>0$), 
and then apply the influence bound of Lemma~\ref{infl_lem}. 
However, even if we choose $n$ very large, 
we cannot avoid that there are intervals with more than one symbol. 
The solution is an intermediate step: for $q'\in(q,q'')$ we couple 
the processes for values $q$ and $q'$ in such a way that `essential' 
symbols have distance at least $\delta$. 
The existence of such a coupling is stated in
Lemma~\ref{stab_lem} below.
Subsequently, we use the influence bound of Lemma~\ref{infl_lem} 
to conclude that when $q'$ is further increased 
to $q''$, then the criterion for percolation
in Lemma~\ref{fin_lem} is satisfied. 

\subsection{Stability coupling}
\label{stab_sec}

The  processes with different values of $q$
can be coupled in a natural way. 
Since this coupling procedure serves as a `starting point' 
for the more complicated coupling
in Lemma~\ref{stab_lem} below, we give a brief sketch here.
Let $\Pi$ be a
Poisson point process with unit density on $\ZZ^2\times\RR$, 
and write $[\Pi]$ for the support of $\Pi$ (i.e. the set of points in a realization of this point process). 
We interpret $(x,t)\in[\Pi]$ as a \emph{symbol} in the graphical 
representation.   In a second step we decide the \emph{type} of the symbol. 
Types are from the set 
\begin{equation*}
{\rm T}=\left\{
\begin{array}{ll}
\big\{ D_1,D_2,U_1,U_2,
A_1^\uparrow,A_1^\downarrow,A_1^\leftarrow,A_1^\rightarrow,
A_2^\uparrow,A_2^\downarrow,A_2^\leftarrow,A_2^\rightarrow\big\} 
&\mbox{(Model A)},\\
\big\{ D_1,D_2,U_1,U_2,
A_1^\uparrow,A_1^\downarrow,A_1^\leftarrow,A_1^\rightarrow\big\} 
&\mbox{(Model B)},
\end{array}\right.
\end{equation*}
corresponding to the notation in Section~\ref{graphical_sec}, arrow 
superscripts indicating the direction of (incoming) arrows. 
From now on we refer to symbols with types $D_1$ and $D_2$
as \emph{down} symbols and the remaining as \emph{up}
symbols.
For each symbol $(x,t)\in[\Pi]$, we consider three independent 
random variables drawn uniformly from the unit interval, 
denoted $Q_{(x,t)}, B_{(x,t)}$, and $G_{(x,t)}$.  These are
independent also of all other random variables used.
We assign an \emph{up} symbol whenever $Q_{(x,t)}\leq q$
and a \emph{down} symbol whenever $Q_{(x,t)}>q$.
For Model A we assign type $D_1$ if 
$Q_{(x,t)}> q$ and $B_{(x,t)}\le \ki/(\ki+\kii)$ and we assign type 
$D_2$ if $Q_{(x,t)}> q$ and $B_{(x,t)}>\ki/(\ki+\kii)$. 
Similarly, whenever $Q_{(x,t)}\le q$, we assign an \emph{up} symbol based 
on the outcome of $G_{(x,t)}$, in such a way that the marginal distributions for 
the ten different \emph{up} symbols (four different arrows of 
type $A_1$, another four of type $A_2$, and the two 
types $U_1$ and $U_2$) have the desired form. 
A very similar construction holds
for Model B (without the $A_2$ symbols). 
We write $H^q$ for the graphical representation thus
obtained.  So $H^q$ consists of the processes  
$D_1, D_2, A_1, A_2, U_1, U_2$ for Model A, 
and $D_1, D_2, A_1, U_1, U_2$ for Model B,
as in Section~\ref{graphical_sec}.
(Of course, $H^q$ depends not only on $q$ but also on the remaining
parameters $\k,\l$, $\kii,\lii$ etc; 
however, we suppress this from the notation.)
The reader may convince her-/himself that the marginal distributions 
coincide with the definition of Section~\ref{graphical_sec}. 

We write $\PP$ for the probability measure governing 
$\Pi$, $Q_{(x,t)}$, $B_{(x,t)}$ and $G_{(x,t)}$ as above,
and $\PP_q$ for the probability measure governing the
resulting graphical representation $H^q$.
Thus $\PP$ is a coupling of all the $\PP_q$'s, $0 \leq q \leq 1$.

For each $x \in \ZZ^2$, $q\in[0,1]$ and $n \in \NN$, 
we define the random variable $\eta_x^{(q,n)}$
as the state ($0$, $1$ or $-1$) at $(x,0)$,
subject to the boundary condition assigning state 1 to any
point $(y,s)$ with $d(x,y)=\lfloor \sqrt{n} \rfloor$ or
$s=-\sqrt{n}$.
(Here $d(x,y)$ denotes the usual graph distance.)
Note that $\eta_x^{(q,n)}$ is determined by the graphical 
representation in the space time region

\begin{equation}\label{eq-st-xqn}
\{(y,s) \, : \, d(x,y) \leq \lfloor \sqrt n \rfloor \mbox{ and } s \in [-\sqrt n, 0]\} . 
\end{equation}
More generally, for $x \in \ZZ^2$, $q\in[0,1]$ and $(z,t)$
in the space-time region \eqref{eq-st-xqn},
we define $\eta_z^{(x;q,n)}(t)$ as the state at $(z,t)$ subject to the same boundary condition as in the definition
of $\eta_x^{(q,n)}$ above
(i.e.  the b.c. assigning state 1 to any
point $(y,s)$ with $d(x,y)=\lfloor \sqrt{n} \rfloor$ or $s=-\sqrt{n}$). \\
Note that $ \eta_x^{(x;q,n)}(0) = \eta_x^{(q,n)}$ and that
$\eta_z^{(x; q,n)}(t) = 1$ if $d(x,z) = \lfloor \sqrt n \rfloor$ or $t = \sqrt n$.

Recall the length $\delta = n^{-\alpha}$ introduced in the 
paragraph preceding this section.
For $v\in \ZZ^2$ and $k\in\NN$, $0\le k\le \lceil \sqrt n/\delta \rceil$, and 
type $\tau\in{\rm T}$, we introduce the indicator functions 
\begin{equation}\label{X_def_eq}
	X_\tau^{(q,k,\delta)}(v):=\mathbbm{1}\big\{
	\text{$\exists$ symbol of type $\tau$ in 
	$\{v\}\times(-k\delta,(-k+1)\delta]$}
	\big\}.
\end{equation}
For each $\delta$-interval,
these $X$ variables only indicate whether there are 
symbols of a certain type in the interval, but do not 
tell us their precise locations or order. 
However, this information is often enough to conclude the value of $\eta_z^{(x;q,n)}(t)$ defined above: 
We define  $\eta_z^{(x;q,n,\delta)}(t)$ as the maximal $m \in \{-1, 0, 1\}$ for which
the values of the elements of
\[
\big\{ X_\tau^{(q,k,\delta)}(v)\colon
\text{$\tau \in T$, $v\in \ZZ^2$, and $k\in\NN$}\big\}
\]
imply that $\eta_z^{(x;q,n)}(t) \geq m$.
(Because of the boundary condition in the definition of 
$\eta_z^{(x;q,n)}(t)$ we can, in fact, restrict to $v$'s with $d(v,x) \leq \sqrt n$ and
$k$'s with $0\le k\le \lceil \sqrt n /\delta \rceil$).
Clearly
$\eta_z^{(x; q,n,\delta)}(t) \leq\eta_z^{(x;q,n)}(t)$ for all $\d>0$.
From now on we write $\eta_x^{(q,n,\delta)}$ for
$\eta_x^{(x;q,n,\delta)}(0)$. (Note that this is the maximal $m \in \{-1, 0, 1\}$ for which
the $X_\tau^{(q,k,\delta)}(v)$'s imply that $\eta_x^{(q,n)} \geq m$). 

Let $L_n$ be the box $[n, 5n] \times [n, 2 n]$. (The precise choice of this box is not essential).
The following result holds for  model A as well as model B (subject to the correct
interpretation);  note that for model A 
we do not require Assumption~\ref{assump} for this result.

\begin{lemma}[Stability coupling]\label{stab_lem}
Let $\alpha>0$ and, for each $n$, let $\delta= \delta_n = n^{-\alpha}$. 
For any $0<q<q'<1$, there is a coupling $\tilde\PP=\tilde\PP_{q,q',n}$ of $\PP_q$ and 
$\PP_{q'}$ such that 
$\tilde\PP\big(\forall x\in L_n:\eta^{(q,n)}_x\le\eta^{(q',n,\d)}_x\big)\rightarrow 1$ 
as $n\to\infty$.
\end{lemma}

We give full details for Model A.
\begin{proof}[Proof for Model A]
Let $R_n$ denote the box $[0, 6n] \times [0, 3 n]$.
Note that $\eta^{(q,n)}_x$ and $\eta_x^{(q',n,\d)}$, $x \in L_n$, are
determined by the graphical representation in the space-time region
\[
\ST_n := R_n\times[-\lceil \sqrt n \rceil,0].
\] 
We let $\delta_1=\sqrt{\delta}=n^{-\alpha/2}$, and throughout 
the proof consider intervals $I$ of the form 
$\{x\}\times[-(k+1)\d_1,-k\d_1]$ whose intersection with 
$\ST_n$ is nonempty ($k\in\NN_0$). 
We call such an interval $I$ \emph{occupied} 
whenever $I\cap[\Pi]\neq\varnothing$. 
Moreover, we call two intervals $I^{(x,k)}$ and 
$I^{(y,\ell)}$ \emph{neighbors} if either $x=y$ 
and $|k-\ell|=1$, or $d(x,y)=1$ and $k=\ell$. 
This neighborhood relation determines the notion of 
\emph{clusters of neighboring occupied intervals}, 
henceforth called \emph{clusters}. 
We use the notation $\cC$ for clusters, and define
the \emph{size} $|\cC|$ to be the number of 
\emph{symbols} in $\cC$ (not the number of intervals). More precisely,
$|\cC| = \sum_{I \in \cC} |I\cap[\Pi]|$, where  $|I\cap[\Pi]|$ is the cardinality
of $I\cap[\Pi]$.
Note that the clusters depend only on the Poisson 
process $\Pi$, hence the law of the clusters is 
independent of any of the parameters $\ki,\kii,\li,\lii,\hi,\hii$, 
and the occupation of the intervals is pairwise independent. 

Let $S^\a_n$ be the event that each occupied cluster in $\ST_n$
has size smaller than $\lceil 9/\alpha \rceil$.
It follows in an elementary way from
properties of the Poisson process (see ~\cite[equation~(18)]{Berg11}) that
\begin{equation}\label{eq_cluster_size_bound}
\lim_{n\to\infty} \PP\big(S^\a_n\big)=1.
\end{equation}

Before proceeding with the argument, here are the main ideas.
Ideally, we would like to take $\alpha$ so large 
that the size of the largest cluster shrinks to 1.
However, later on (just before \eqref{q2_eq_2}) 
we need to take $\alpha$ rather close to 0; 
hence
the existence of clusters of size $\ge2$ cannot be ruled out, 
and a result of the form \eqref{eq_cluster_size_bound} is
essentially  the best bound we get. 
(The precise value $9 / \alpha$ is not important).
We solve the issue by factorizing our probability space 
$\Omega=\Omega_1\times\Omega_2$, where $\Omega_1$ determines the
clusters of intervals, the relative order (w.r.t.\!
the time coordinates) of the symbols within each cluster, 
as well as some other information, and $\Omega_2$ is responsible for 
the `fine-tuning' (including the precise location of the symbols). 
We then first sample from $\Omega_1$, which in particular fixes the clusters. 
For each cluster, when sampling from $\Omega_2$, we use a 
`crossover', which sacrifices an unnecessarily `good' event in order to 
avoid a `bad' event (where `bad' means lack of 
$\delta$-stability). `Crossover' techniques 
have been used earlier for Voronoi percolation in~\cite[Theorem~6.1]{BolloRiord06},
and for percolation in the (ordinary) contact process  in~\cite{Berg11}. 
However, 
it turns out that the model we consider requires a considerably 
more subtle `crossover recipe' than
in~\cite{Berg11}.

We now give a detailed description.
Outcomes $\om_{1}\in\Omega_{1}$ contain the following
partial information 
about $\Pi$: First of all, for any interval $I$, $\om_1$ determines the number of elements 
of $[\Pi]\cap I$.  This identifies the clusters of $\ST_n$. 
We call an interval $I^{(x,k)}=\{x\}\times[-(k+1)\d_1,-k\d_1]$ 
\emph{vertically isolated} whenever $I^{(x,k)}$ contains precisely one
symbol and both $I^{(x,k-1)}$ and $I^{(x,k+1)}$ are not occupied. 
Further, for any cluster $\cC$, we let $\om_{1}$ also determine the 
relative order of symbols in $\mathcal C$ (w.r.t. the time coordinates of the symbols),
and the value of $G_{(x,t)}$ for all symbols in $\cC$.
(For the ease of description we here name symbols by $(x,t)$ 
although the precise time $t$ is not yet determined. Further,
recall that $G_{(x,t)}$ tells {\em which} 
up  type a symbol has {\em if} its 
type is  up).
Finally, we also let $\om_{1}$ determine the value of $B_{(x,t)}$ for symbols in 
\emph{vertically isolated intervals only}. 

Outcomes $\om_2\in\Omega_2$ determine the precise location of
symbols $(x,t)\in\Pi$, as well as the value $Q_{(x,t)}$ for
\emph{all} $(x,t)\in[\Pi]$, and the value of $B_{(x,t)}$ for every $(x,t)$ that is
{\em not contained in a vertically isolated interval}. 
Write $\cF_{1}$ and $\cF_{2}$ for the corresponding 
$\s$-algebras on $\Om_{1}$ and $\Om_{2}$
respectively.

Following the discussion at the beginning of the section, we can 
obtain the graphical representation $H^q_n$ of the process 
on the space-time box $\ST_n$ as function of
$\om_1$, $\om_2$, and $q$: 
\begin{equation}
\label{coup_def_eq}
H^q_n=H^q_n(\om_1,\om_2).
\end{equation}
(Actually, $H^q_n$ depends also on the remaining parameters
$\ki,\li,\hi$ etc but we suppress this dependence.)
Since $\eta^{(q,n)}_x$ is itself a monotone (in $q$) function of 
$H^q_n$,
$\PP$ gives a coupling of $\eta^{(q,n)}_x$ and $\eta^{(q',n)}_x$ such 
that $\eta^{(q,n)}_x\leq\eta^{(q',n)}_x$ for any $q<q'$.
The restriction of  $H^q_n$ to a cluster $\cC$ is denoted  
$H^q_{n,\cC}=H^q_{n,\cC}(\om_1,\om_2)$.

The `crossover' referred to above is a mapping
$\Om_2\rightarrow\Om_2$ which depends on the outcome
of $\om_1\in\Om_1$.
To this end, fix an instance $\omega_1\in\Omega_1$
(which, as mentioned above, particularly fixes the clusters). 
Since $\PP(\,\cdot\mid\cF_1)$ acts independently on the different
clusters, we can write
$\PP(\,\cdot\mid\cF_1)=\prod_{\cC}\PP_\cC(\,\cdot\mid\cF_1)$.
Fix a cluster $\cC$ and 
write $\PP^{\om_1}_\cC(\,\cdot)=\PP_\cC(\,\cdot\mid\cF_1)(\om_1)$.
In light of~\eqref{eq_cluster_size_bound}, we proceed under the
assumption that $|\cC|<\lceil 9/\alpha \rceil$. 
On the probability space $\Omega_2$ we now define two events. (The definition of these
events involves a cluster $\cC$ and hence also $\omega_1$; however, recall that we consider $\omega_1$ as fixed here.)
The first event is 	
\begin{quote}
$\cB$: in $H^q_{n,\cC}$ there are two symbols whose time coordinates
differ by less than $\delta=n^{-\alpha}$.
\end{quote}
The probability of $\cB$ is maximized when all $|\mathcal C|$ symbols
are in one single interval, so that
\begin{equation}
\label{eq-Bbd}
\PP^{\omega_1}_{\cC}(\cB)\le |\cC|^2\,\frac{2\delta}{\delta_1}\le
2{\lceil 9/\alpha \rceil}^2\,n^{-\alpha/2},\quad\mbox{if }
|\cC|<\lceil 9/\a\rceil,
\end{equation}
which goes to $0$ as $n\to\infty$. 

Before we state the other event, we need the following notion: 
A {\em maximal connected vertical chain}
is a union of occupied $\delta_1$-intervals $I^{(x,k)}$, 
$I^{(x,k+1)},\dotsc, I^{(x,k+m-1)}$, with
$k \geq 0$, $m \geq 1$, and where  $I^{(x,k+m)}$ and (in case $k \geq 1$) $I^{(x,k-1)}$ are vacant.
We call $m$ the length of the chain. Note that a vertically isolated interval (defined earlier) is a 
maximal connected vertical chain of length $1$.

We now define the event $\cG$ that: 
\begin{enumerate}
\item  in $H^q_{n,\cC}$ all 
symbols  are \emph{down} symbols (i.e., all symbols in $\cC$ have $Q$-value larger than $q$),
\item in $H^q_{n,\cC}$,
 each maximal connected vertical chain of length $\geq 2$
  has lowest symbol of type $D_1$ and all
other symbols of type $D_2$, and
\item in $H^{q'}_{n,\cC}$ all symbols are \emph{up} symbols
(i.e., all symbols in $\cC$ have $Q$-value smaller than $q'$).
\end{enumerate}
From the above definitions it follows straighforwardly that
\begin{equation}\label{prob_G_eq}
\PP^{\omega_1}_{\mathcal C}(\cG)\ge 
(q'-q)^{|\mathcal
  C|}\big(\min\{\ki/(\ki+\kii),\kii/(\ki+\kii)\}\big)^{|\cC|}.
\end{equation}

By this and \eqref{eq-Bbd} we thus may choose $n$ sufficiently large 
(not depending on $\mathcal C$ or otherwise on $\omega_1$) such that 
\begin{equation}\label{prob_GB_eq}
	\PP^{\omega_1}_{\mathcal C}(\cG)\ge \PP^{\omega_1}_{\mathcal C}(\cB)
	\qquad\text{if $|\mathcal C|< \lceil 9/\alpha \rceil$}.
\end{equation}

Write $\cB'=\cB\setminus \cG$.
From the above we get, for
$n$ sufficiently large, that if $|\mathcal C|< \lceil 9/\alpha \rceil$ then
there exists a measurable subset $\cG'\subset \cG\setminus \cB$ and 
a measure-preserving 1-1 map $\psi_{\mathcal C}$ 
on $\Omega_2$ such that 
\begin{itemize}
\item $\psi_{\cC}(\cB')=\cG'$,
\item $\psi_{\cC}(\cG')=\cB'$, and
\item  $\psi_{\cC}(\omega_{2})=\omega_2$ whenever 
$\omega_2\not\in \cB'\cup \cG'$. 
\end{itemize}
If, on the other hand, $|\mathcal C|\ge \lceil 9/\alpha \rceil$, then
we let $\psi_{\mathcal C}$ be the identity on $\Omega_2$. 

The map $\psi_\cC$ is the crossover mentioned before.
Since $\psi_{\mathcal C}$ is measure-preserving on $\Omega_2$, we 
obtain a new coupling of the graphical representations on 
the cluster $\mathcal C$ by considering the graphical 
representation 
$$\tilde H_{n,\cC}^{q'}:=H_{n,\cC}^{q'} (\om_1,\psi_{\cC}(\om_2)),$$
cf.\ \eqref{coup_def_eq}. 
The `overall coupling' is then obtained by constructing
$\tilde H_{n,\cC}^{q'}$ for each cluster ${\cC}$ independently.
The resulting graphical representation is denoted $\tilde H_{n}^{q'}$. 
We construct $\eta^{(q,n)}_x$ from the graphical 
representation $H_{n}^{q}=H_{n}^{q} (\omega_1,\omega_2)$, 
and $\eta^{(q',n,\d)}_x$ from $\tilde H_{n}^{q'}$. 

Finally, we check that this coupling has the desired properties. 
For a given $\omega_1$, let $\mathcal C$ be one of the clusters.
Recall that $H_{n,\cC}^{q'}\geq H_{n,\cC}^{q}$.
We have to study $\tilde H_{n,\cC}^{q'}$ and compare it with $H_{n, \cC}^q$. These objects depend on $\om_2$.
There are three cases: 
\begin{enumerate}
\item[(i)] Case $\omega_2\not\in(\cB\cup \cG)$. 
This case is simple: by the definition of the coupling
procedure, we have
$\psi_{\mathcal C}(\omega_2)=\omega_2$, and  
$\tilde H_{n,\cC}^{q'} = H_{n,\cC}^{q'}$, which (as we recalled above) dominates $H_{n,\cC}^{q}$.
Moreover, 
since $\omega_2\not\in \cB$ we know that $H_{n,\cC}^{q}$, and 
hence also $\tilde H_{n,\cC}^{q'}$, does not have two symbols of which the
time coordinates differ less than $\delta$. This settles case (i).
\item[(ii)]  Case $\omega_2\in \cB'$.
Then, by the definition of the coupling procedure, 
$\psi_{\mathcal C}(\omega_2)\in \cG'\se \cG\setminus \cB$.  
By the definition of $\cG$, this implies that all symbols in
$\tilde H_{n,\cC}^{q'}$ are up symbols. Since the precise type of an up symbol is determined by
$\omega_1$, we get that each symbol in $\tilde H_{n,\cC}^{q'}$ `dominates' the corresponding
symbol in $H_{n,\cC}^{q}$. Moreover, since 
$\psi_{\mathcal  C}(\omega_2)$ is not in $\cB$, 
there are no symbols in $\tilde H_{n,\cC}^{q'}$ of which the 
time coordinates differ less than $\delta$.
Finally, the {\em order} (w.r.t.\ time) 
of the symbols in $\tilde H_{n,\cC}^{q'}$ is the same as for $H_{n,\cC}^{q}$
(recall that the order is determined by $\omega_1$). This settles case (ii).
\item[(iii)] Case $\omega_2\in \cG$.
Then, by the definition of $\cG$, the types in
$H_{n,\cC}^{q}$  
on the maximal connected vertical chains that are not
single vertically isolated intervals, are as `unfavourable' as
possible: 
Consider such a chain
and let $I^{(z,k)}$ be its
`highest' (i.e., with largest time index) interval. 
Since the symbol with smallest time coordinate on the chain
has type $D_1$ and the others $D_2$, and since there are no 
incoming arrows, it follows that,
for all vertices $x$ for whose state at time $0$
this part of space-time is `relevant' 
(i.e. for all $x \in L_n$ with $d(x,v) \leq \sqrt n$),  
$$\eta_z^{(x;q,n)}(-k \delta) = -1.$$
Further, each single vertically 
isolated interval has (by the definition of $\cG$)
in $H_{n,\cC}^{q}$ a `down' symbol. Since the precise type of 
this down symbol is determined by $\omega_1$,
it follows that the corresponding symbol in $\tilde H_{n,\cC}^{q'}$ is 
either the same type of down symbol,
or an up symbol. From these considerations it follows that, no 
matter how the symbols in $\tilde H_{n,\cC}^{q'}$
are located precisely, we have that, if $\cC$ would be the only cluster,
then
each space-time point $(z,t)$ which is the `highest point'
of a maximal connected vertical chain of $\cC$, satisfies 
\begin{equation} \label{eq-r-eta-ineq}
\eta_z^{(x;q',n,\delta)}(t) \geq \eta_z^{(x;q,n)}(t) \mbox{ for all } x \in L_n \mbox{ with } d(x,v) \leq \sqrt n.
\end{equation}
\noindent
This settles the last case.
\end{enumerate}

At the end of case (iii) we stated that \eqref{eq-r-eta-ineq} would hold if $\cC$ is the only cluster. In fact,
by combining this statement with
the conclusions concerning case (i) and (ii), and the monotonicity of the contact process dynamics, it follows
that \eqref{eq-r-eta-ineq} also holds (for such $(z,t)$) if there are other clusters (as long as all clusters have size
$\leq 9/\alpha$).  This
completes the proof of Lemma \ref{stab_lem} for 
Model A.
\end{proof}

\begin{proof}[Sketch proof for Model B]
The argument for Model B has the same
structure, but the details are considerably simpler.  
One difference is that now the only information represented by $\omega_1$ is the number of symbols in each interval (which in turn defines the clusters) and the values of $U_{(x,t)}$. 
All other information (the precise locations of the symbols and the $Q_{(x,t)}$- and $B_{(x,t)}$-values) are represented by $\omega_2$.
Another difference is that we modify the definition of the event $\cG$ to the following:
\begin{enumerate}
\item  in $H^q_{n,\cC}$ all symbols are of type $D_2$;
\item in $H^{q'}_{n,\cC}$ all symbols are \emph{up} symbols.
\end{enumerate}
This implies $Q_{(x,t)}\in(q,q')$ for all symbols in $H^q_{n,\cC}$. 
In particular, no special `treatment' of maximal connected vertical chains is needed anymore.
Equation~\eqref{prob_G_eq} becomes
\[
\PP^{\om_1}_\cC(\cG)\geq
(q'-q)^{|\cC|} \big(\ki/(\ki+\kre)\big)^{|\cC|},
\]
so that~\eqref{prob_GB_eq} still holds
for large enough $n$.  Thus we may define  crossover
maps $\psi_\cC$ and the modified graphical representation
$\tilde H^q_n$ as before.  To check that this coupling
has the required properties, we distinguish again the 
three cases (i)--(iii) as
for the 3-state model. 
Indeed, the arguments for cases (i) and (ii) apply verbatim as in the 3-state case. 
Case (iii) is now considerably simpler than before, because $\omega_2\in\cG$ implies that$H^q_{n, \cC}$ has only $D_2$ symbols and hence ${\tilde H}^{q'}_{n, \cC}$ is always `at least as good'.
\end{proof}

\subsection{Proof of Theorem~\ref{3st_sharp_thm}}
\label{sharp_pf_sec}

Let $\ui_q$ denote the upper invariant measure for the
contact process defined as above with
parameter value $q$.  Let $0<q_1<1$ be such that
under $\ui_{q_1}$ the cluster size $|C_0|$
does \emph{not} have
exponential tails.  Let $q_2 > q_1$. We will deduce that
$\ui_{q_2}(|C_0|=\oo)>0$. This immediately implies 
Theorems~\ref{3st_sharp_thm}.

By the first part of Lemma~\ref{fin_lem} and Lemma~\ref{rsw_lem}
we have that there exists $\eps_1>0$ and 
a sequence $n_i\rightarrow\oo$ such that 
\[
\ui_{q_1}(H(4n_i,n_i))\geq\eps_1\quad
\mbox{ for all } i\geq1.
\]
Recall that $L_n$ denotes the $4n$-by-$n$
rectangle $[n,5n]\times[n,2n]$, and
write $H_i=H(L_{n_i})$.
Recall the definition of $\eta^{(q,n)}$ from the paragraphs preceding and below \eqref{X_def_eq}.
By monotonicity (cf.\ Lemma~\ref{3st_mon_lem} 
and the discussion around there)
the realization of $\ui_q$ on $L_n$ is dominated by
$\eta^{(q,n)}$ for all $n\geq1$.  We deduce that
\begin{equation}\label{eq-r-etep}
\PP(\eta^{(q_1,n_i)}\in H_i)\geq\eps_1\quad
\mbox{ for all } i\geq1.
\end{equation}
Fix  $q'\in(q_1,q_2)$.  By Lemma~\ref{stab_lem} and \eqref{eq-r-etep},
\[
\PP(\eta^{(q',n_i,\d)}\in H_i)
\geq\eps_2:=\eps_1/2
\]
for all large enough $i\geq1$.
The latter probability is defined in terms of the
Bernoulli variables $X$ of~\eqref{X_def_eq}, so
in principle Lemma~\ref{infl_lem} could now be applied.
However, we have no good way of bounding the number
$N$ of variables with maximal influence.  To get around this,
we consider a `symmetrized' version of the event $H_i$.
A similar method was used in e.g. ~\cite{BolloRiord06} and
~\cite{Berg11} and is standard in this type of
argument;  
here we use the `truncation'
implicit in the definition of the $\eta^{(q,n,\d)}$
and hence, ultimately, the fast convergence of the 
dynamics (Lemma~\ref{exp_lem} and Assumption~\ref{assump}).

Recall that $R_n$ is the box $[0,6 n] \times [0, 3 n]$, and
consider the `periodic' set $R_n^{\mathrm{per}}$
obtained from $R_n$ by identifying the left and right
sides;  that is, identifying points $(6n,y)$
and $(0,y)$.  We can consider $L_n$ as a subset
of $R_n^{\mathrm{per}}$ rather than $\ZZ^2$.  Since the variables
$\eta^{(q',n_i,\d)}$ are `truncated' at distance 
$\sqrt{n_i}$ the probability that
$\eta^{(q',n_i,\d)}\in H_i$ is (for large enough $i$)
unchanged under this change of geometry.  Let $A_i$
be the event that there is a horizontal crossing of 1's
of at least one of the $6n_i-1$ horizontal translates of
$L_{n_i}$ in $R_{n_i}^{\mathrm{per}}$.  Thus
\begin{equation}\label{pi_i_eq}
\pi_i(q'):=\PP(\eta^{(q',n_i,\d)}\in A_i)\geq
\PP(\eta^{(q',n_i,\d)}\in H_i)
\geq\eps_2,
\end{equation}
for all sufficiently large $i$.

We apply Lemma~\ref{infl_lem} to the event $A_i$.
By symmetry, all $6n_i-1$ horizontal translates
of $X^{(q',k,\d)}_\tau(v)$ have the same influence,
so the number $N$ of Lemma~\ref{infl_lem}
satisfies $N\geq 6n_i-1\geq n_i$.  The number
$m$ of that lemma corresponds to the number of
different types $\tau$ where we don't 
distinguish between different directions of arrows 
(because the Poisson intensities do not depend
on these directions).
Thus $m=6$
for Model A, and $m=5$ for Model B.  
(The number $n$ of variables
of each type which appears in that lemma does
not figure in the conclusion, so it is irrelevant for us.)

For the next step of the argument, we consider the two models separately. 
Consider Model A first.
We let $p_1,\dotsc,p_6$ denote the probabilities 
that $X^{(q',k,\d)}_\tau(v)$ equals 1 for $\tau=U_1$, 
$\tau=U_2$, $\tau \in A_2$,
$\tau\in A_1$, $\tau=D_1$, and
$\tau=D_2$, respectively.  Recall that
$\ki+\kii+4\li+4\lii+\hii+\hii=1$ and that we increase 
$q=4\li+4\lii+\hi+\hii$ while keeping $\ki/\kii$
and the ratios between any two of $\lii, \li, \hii, \hi$ fixed.  
This implies that there are constants $r_1,\dotsc,r_6\in(0,1)$
such that $\hi=r_1q$, $\hii=r_2q$,
$\li=r_3q$, $\lii=r_4q$, $\ki=r_5(1-q)$, and $\kii=r_6(1-q)$.
Hence $p_j$ equals $1-e^{-r_jq\d}$ for $1\leq j\leq4$,
and $1-e^{-r_j(1-q)\d}$ for $j=5,6$.  It follows that
for $i$ large enough
\begin{equation}\label{eq-newlabel1}
\begin{split}
\frac{d\pi_i}{dq}&=
\sum_{j=1}^4 \d r_je^{-r_j\d q}\frac{\partial\pi_i}{\partial p_j}-
\sum_{j=5}^6 \d r_je^{-r_j\d (1-q)}\frac{\partial\pi_i}{\partial p_j}\\
&\geq
\d C\Big[
\sum_{j=1}^4 \frac{\partial\pi_i}{\partial p_j}-
\sum_{j=5}^6 \frac{\partial\pi_i}{\partial p_j}\Big]\\
&\geq \d C \log N
\frac{\pi_i(1-\pi_i)}{K'\d\log(2/\d)},
\end{split}
\end{equation}
for some constants $C, K'$, and where the last inequality comes 
from Lemma~\ref{infl_lem}.

Let $\eps_3>0$ and suppose that $\pi_i(q'')<1-\eps_3$
for all $q''\in(q',q_2)$.  Using that $N\geq n_i$
and $\d=n_i^{-\a}$ we deduce from~\eqref{pi_i_eq} and \eqref{eq-newlabel1} that
$\pi_i(q_2)\geq C(q_2-q')\eps_2\eps_3/\a$.
Choosing $\a$ suffiently small we reach the following
conclusion:
\begin{equation}\label{q2_eq_2}
\forall\eps^*>0\,\exists\a>0:
\mbox{ for large enough } i, \; \pi_i(q_2)\geq1-\eps^*.
\end{equation}
For Model B, we obtain \eqref{q2_eq_2} in 
literally the same way, except that $\lii=r_4=p_4=0$ 
(because there are no $A_2$ symbols) and $\kii=\kre$. 

This final argument is the same for both models.  
Note that the event $A_i$ implies that there is a
horizontal crossing of at least one of the following rectangles
(regarded as subsets of $R_n^{\mathrm{per}}$):
\[
[jn_i,(j+3)n_i\,(\mbox{mod }6n_i))]\times [n_i,2n_i]\quad
0\leq j\leq 5.
\]
Thus (by using the \textsc{fkg}-inequality)
for $\hat\eps>0$ as in Lemma~\ref{fin_lem},
\[
\PP(\eta^{(q_2,n_i,\d)}\in H(3n_i,n_i))\geq 1-
(1-\PP(\eta^{(q_2,n_i,\d)}\in A_i))^{1/6}\geq 1-\hat\eps/2
\]
for all sufficiently large $i$, where the last inequality comes from 
\eqref{q2_eq_2}.
The family of random variables $(\eta^{(q_2,n_i,\d)}_x:x\in[n,4n]\times[n,2n])$ is 
clearly stochastically dominated by the family $(\eta^{(q_2,n_i)}_x:x\in[n,4n]\times[n,2n])$, and the law
of this latter family has (by Lemma~\ref{exp_lem} and Assumption~\ref{assump}, and using standard
arguments) total variation distance at most 
\[
C_1'\cdot 3n_i^2\exp\big(-C_2'\sqrt{n_i}\big)
\]
from $\ui_{q_2}$.  Hence $\ui_{q_2}(H(3n_i,n_i))\geq 1-\hat\eps$
for large enough $i$, which by Lemma~\ref{fin_lem}
implies that $\ui_{q_2}(|C_0|=\oo)>0$.
This completes the proof of Theorem~\ref{3st_sharp_thm}.
\qed

\subsection{Proof of Corollary~\ref{3st_sharp_cor}}
\label{sharp_prop_sec}
Recall that
for $x,y\in\RR^k$ we write $x\prec y$
if each coordinate of $x$ is strictly
smaller than the corresponding coordinate of $y$.
We write $\RR_+=(0,\oo)$.
When proving Corollary~\ref{3st_sharp_cor} we make
use of the following fact.

\begin{lemma}\label{dilworth_lem}
Let $n\geq1$ be fixed and let
$B\se\RR_+^3$ be a measurable set with 
the following property:
\begin{equation}\label{property}
\mbox{if } a_1,a_2,\dotsc,a_m\in B
\mbox{ satisfy } a_1\prec a_2\prec \dotsb \prec a_m
\mbox{ then } m\leq n.
\end{equation}
For $y\in\RR_+^2$ write $B(y)=(\RR_+\times\{y\})\cap B$
and let $\G$ denote the set of $y\in\RR_+^2$
such that $B(y)$ is uncountable.
Then $\G$ has measure zero.
\end{lemma}
\begin{proof}
The partially ordered set $(B,\prec)$ has
height at most $n$, so by (the dual version
of) Dilworth's theorem, $B$ can be partitioned
into $n$ antichains.  An antichain in this situation
is a set satisfying~\eqref{property}
with $n=1$, so it suffices to consider that case.

For $n=1$, note that if $x<x'$ and both 
$(x,y)$ and $(x',y)$ belong to $B$ then 
property~\eqref{property} is preserved if the 
interval $[x,x']\times\{y\}$ is added to $B$.  
Thus we may assume that $B$ is maximal in the sense
that it includes all such intervals.  With each 
$y\in\G$ we may thus associate a rational number
$q(y)$ such that $(q(y),y)$ lies in an interval of
$B(y)$.  We now write $y\in\RR_+^2$ in polar
coordinates $(\theta,r)$ with $\theta\in(0,\pi/2)$
and $r>0$.  Fix $\theta$ and $r<r'$ and write
$y=(\theta,r)$ and $y'=(\theta,r')$.  
If $(x,y)\in B(y)$ and $(x',y')\in B(y')$
then $x'\leq x$,
by~\eqref{property} with $n=1$.  Thus, if $y,y'\in \G$
then we may choose $q(y')<q(y)$.  It follows that
for each $\theta\in(0,\pi/2)$ the set of $r>0$
such that $(\theta,r)\in\G$ is at most
countable.  By Fubini's theorem (using polar
coordinates) it follows that $\G$ has measure zero.
\end{proof}

\begin{proof}[Proof of Corollary~\ref{3st_sharp_cor} for Model A]
Fix arbitrary $\ki,\kii>0$.
Recall that $\hi_\perc$ is decreasing
in each of the parameters $\li,\lii,\hii$.
Since $\hi_\perc(\li,\lii,\hii)<\oo$ 
for $\hii>0$ we may restrict
$(\li,\lii,\hii)$ to one of the (countably many)
sets where $\hi_\perc$ is bounded above by 
a fixed integer $K$.
We call a triple $(\li,\lii,\hii)$ `bad'
if there is some $\d>0$ such that 
\begin{equation}\label{bad_eq}
\hi_\perc(\li',\lii',\hii')\leq
\hi_\perc(\li,\lii,\hii)-\d\mbox{ for all }
(\li',\lii',\hii')\succ(\li,\lii,\hii).
\end{equation}
If $B$ denotes the set of `bad' points, then
we may write $B=\cup_{n\geq 1}B_n$, where
$B_n$ is the set of points such 
that~\eqref{bad_eq} holds with $\d=K/n$.
The set $B_n$ satisfies~\eqref{property}, so
it follows from Lemma~\ref{dilworth_lem} that
for almost all pairs $(\lii,\hii)$ the
set of $\l>0$ such that $(\li,\lii,\hii)$
is bad is countable.  

We are therefore done if
we show that if $(\li,\lii,\hii)$ is \emph{not}
bad, and $\hi<\hi_\perc(\li,\lii,\hii)$,
then the cluster size decays exponentially
for the parameter values 
$\ki,\kii,\li,\lii,\hi,\hii$.  Writing
$\d=\hi_\perc(\li,\lii,\hii)-\hi$ we have that
there exists $(\li',\lii',\hii')\succ(\li,\lii,\hii)$
such that 
\[
\hi_\perc(\li',\lii',\hii')>
\hi_\perc(\li,\lii,\hii)-\d=h. 
\]
The result now follows from 
Theorem~\ref{3st_sharp_thm}.
\end{proof}

\begin{proof}[Proof of Corollary~\ref{3st_sharp_cor} for Model B]
Fix $\k,\kre>0$.  
The argument for Model B is very similar to the one already
given for Model A, using the analog of 
Lemma~\ref{dilworth_lem} with $B\se\RR_+^2$
(in which case we can actually show that $\G$ is 
at most countable). 
One small difference is that
we may now have $h_\perc=\oo$ for 
some $(\li,\hii)$.  We now say that
$(\li,\hii)$ is \emph{bad} if $h_\perc(\li,\hii)<\oo$
and there exists $\delta>0$ such that
\begin{equation}\label{bad_eq_2}
\hi_\perc(\li',\hii')\leq
\hi_\perc(\li,\hii)-\d\mbox{ for all }
(\li',\hii')\succ(\li,\hii),
\end{equation}
and that $(\li,\hii)$ is 
\emph{terrible} if $h_\perc(\li,\hii)=\oo$
but $h_\perc(\li,\hii)<\oo$
for all $(\li',\hii')\succ(\li,\hii)$.
The set $B$ of bad points may be written as
$$B=\bigcup_{K\geq1}\bigcup_{n\geq1} B^{(K)}_n,$$
where $B^{(K)}_n$ is the set of 
points $(\li,\hii)$ such that 
$\hi_\perc(\li,\hii)\leq K$ 
and~\eqref{bad_eq_2} holds with $\delta=K/n$.
Thus $B^{(K)}_n$ satisfies~\eqref{property}.
The set $T$ of terrible points 
satisfies~\eqref{property}
with $n=1$.  It follows that for almost all
$\hii>0$, the set of $\li$ such that
$(\li,\hii)\in B\cup T$ is at most countable.
If $(\li,\hii)\not\in B\cup T$ then
the result follows in the same way
as for Model A.
\end{proof}

%%%%%%%%%%%%%%%%%%%%%%%%%%%%%%%%%%%%%%%%%%%%%%%%%%%

\section{Existence of density-driven processes}
\label{ddcp_sec}

In this section
we prove the existence of {\ddcp} (Definition~\ref{ddcp_def})
using a fixed-point argument.  Although this result is strictly
speaking not needed for our main results on sharpness and lack of
robustness (since,
as discussed in Section~\ref{stat_sec}, 
\emph{stationary} {\ddcp} are simply contact processes with
constant parameters), we find it interesting in itself.

We consider 3-state processes with constant $\k,\kii,\lii,\hii$.
Recall that we write $\rho(t)=P(X_0(t)=1)$
for the density of the process.
We let $L^\oo_b$ denote the
set of measurable $h\colon[0,\oo)\rightarrow[0,\oo)$ which are bounded on each
compact subinterval.

We prove the following 
existence result for the Model A;
a completely analogous result holds for Model B.
\begin{theorem}\label{3st_constr_thm}
Let $\dfl,\dfh\colon[0,1]\rightarrow[0,\oo)$ be uniformly Lipschitz continuous.
For all $\ki, \kii, \lii,\hii \geq 0$ and each 
translation-invariant probability measure $\nu$ on
$\{-1,0,1\}^{\ZZ^d}$, there is a unique pair $(\li, \hi) \in
L^\oo_b\times L^\oo_b$ such that
the 3-state contact process (Model A) with initial distribution $\nu$
and parameters $\ki, \kii, \li(\cdot),\lii, \hi(\cdot),\hii$ satisfies
$\li(t) = \dfl(\rho(t))$ and $\hi(t) = \dfh(\rho(t))$ for all $t \geq 0$.
\end{theorem}

\begin{proof}
Let $\hi ,\hi',\li,\li' \in L^\oo([0,\oo),[0,\oo))$, and
let $D_1$, $D_2$, $U_2$ and $A_2$ be as in 
Section~\ref{graphical_sec}. 
The intensities of these processes are kept fixed.
Let $\ul U_1$ be a Poisson process of
intensity $\hi(t)\wedge \hi'(t)$. Let $\tilde U_1^{(\hi)}$ and $\tilde U_1^{(\hi')}$ denote
independent Poisson processes (independent also of $\ul U_1$) with
intensities $\hi(t)-(\hi(t)\wedge \hi'(t))$ and $\hi'(t)-(\hi(t)\wedge \hi'(t))$, respectively.  Write
$U_1^{(\hi)}=\ul U_1\cup\tilde U_1^{(\hi)}$, $U_1^{(\hi')}=\ul U_1\cup\tilde U_1^{(\hi')}$ and
$\ol U_1=\ul U_1\cup\tilde U_1^{(\hi)}\cup\tilde U_1^{(\hi')}$.
In the same way (and independently 
of the Poisson processes above) we define $\ul A_1$,
$A^{(\li)}_1$, $A^{(\li')}_1$ and $\ol A_1$. 
Furthermore, let $m := 1 \vee \sup\{ \l(t)\vee\l'(t):t\geq0\}$,
and let $A_1^{(m)}$ be obtained from $\ol A_1$ by appending
another independent Poisson process of intensity 
$m-(\li(t)\vee\li'(t))$.  Note that an element of $A_1^{(m)}$
(at time coordinate $t$)
belongs to $\ol A_1\setminus\ul A_1$ with probability
$|\li'(t) - \li(t)| / m\leq \|\li'-\li\|_\infty$. 

Let $\ul X$ be the contact process 
with $0\rightarrow1$ transitions given by
$\ul U_1$ and $\ul A_1$
(and remaining transitions given by $D_1$, $D_2$, $U_2$, $A_2$).
Similarly,
$X^{(\hi,\li)}$, $X^{(\hi',\li')}$, and $\ol X$ denote the contact processes
with $0\rightarrow1$ transitions given by
$U^{(\hi)}_1$ and $A^{(\li)}_1$,
with $U^{(\hi')}_1$ and $A^{(\li')}_1$, and with $\ol U_1$ and $\ol
A_1$,  respectively.
The construction above is done such that
$\ul X\leq X^{(\hi,\li)}, X^{(\hi',\li')}\leq\ol X$ holds and 
$\ol U_1\setminus\ul U_1$ has rate $|\hi(t)-\hi'(t)|$.

Now consider,  for each $t \geq 0$,  the set
$\bac_t$, which is defined as the set of space-time points
$(x,s)$, $0 \leq s \leq t$, such that there is a
space-time path from $(x,s)$ to $(0,t)$ using arrows from 
$A_1^{(m)}\cup A_2$.
Further, let 
$|\bac_t|$ be the sum of the  total
Lebesgue measure of all intervals constituting
$\bac_t$, plus the number
of arrows in the space--time paths
in the definition of $\bac_t$. 

Let $B_t$ be the event that none of the arrows in the space-time paths in the definition of $\bac_t$ belongs
to $\ol A_1 \setminus \ul A_1$.
If $B_t$ occurs and $(\ol U_1\setminus \ul U_1)\cap \bac_t = \es$,
then $\ol X_0(t)=\ul X_0(t)$ and hence
$X_0^{(\hi,\li)}(t) = X_0^{(\hi',\li')}(t)$.

Equip $L^\oo([0,\oo),[0,\oo))^2$ with the norm
$\|(\li,\hi)\|=\|\li\|_\oo+\|\hi\|_\oo$, and consider the mapping $R$
from this space to $L^\oo([0,\oo),[0,1])$ given by letting 
$R(\li,\hi)(t)=P(X_0(t)=1)$ where $X$ is the 3-state contact process
with rates $\ki, \kii, \li(\cdot),\lii, \hi(\cdot),\hii$.  Let $\alpha > 0$. By the above we have,
for all $0\leq t\leq\a$,
\[
\begin{split}
|R(\li,\hi)(t) - R(\li',\hi')(t)|&\leq 
P(B_t^c \mbox{ occurs  or }
(\ol U_1\setminus \ul U_1)\cap \bac_t\neq\es)\\
&\leq E|\bac_t|(\|\li-\li'\|_\oo + \|\hi-\hi'\|_\oo),
\end{split}
\]
which by obvious monotonicity is at most
$$E|\bac_\a|(\|\li-\li'\|_\oo + \|\hi-\hi'\|_\oo).$$
Let $K_1,K_2$ be uniform Lipschitz constants for $\dfl,\dfh$.  It
follows that 
\begin{multline*}
\sup_{0\leq t \leq \alpha} |\dfl(R(\li, \hi)(t))-\dfl(R(\li',\hi')(t))|
 + \sup_{0 \leq t \leq \alpha} |\dfh(R(\li,\hi)(t))-\dfh(R(\li',\hi')(t))| \\ \leq
(K_1+K_2) \, E|\bac_\a| \, (\|\li-\li'\|_\oo+\|\hi-\hi'\|_\oo).
\end{multline*}

By standard comparison with a branching process, 
it is easy to see that $E|\bac_\a|$ 
is finite for $\a$ sufficiently small, 
and goes to $0$ as $\alpha \rightarrow 0$.
Hence there is an $\a_0 > 0$ such that the mapping 
$\G=\G^\nu_{\a_0}:(\li,\hi)\mapsto (\dfl(R(\li,\hi)),\dfh(R(\li,\hi)))$
is a contraction of 
$L^\oo([0,\a_0],[0,\oo))^2$.  By Banach's
fixed point theorem, this gives the desired result 
for the time interval $[0, \a_0]$. By repeating (`concatenating') this 
result, it can be extended to $[0, 2 \a_0]$, $[0, 3 \a_0]$, etcetera, which
completes the proof.
\end{proof}

\subsection*{Acknowledgement}

The authors  thank Demeter Kiss for helpful and
interesting discussions about Lemma~\ref{dilworth_lem},
and the anonymous referees for helpful comments.

%%%%%%%%%%%%%%%%%%%%%%%%%%%%%%%%%%%%%%%%%%%%%%%%%%

\bibliographystyle{plain}

\begin{thebibliography}{99}

 
\bibitem{aiz-gr-strict}
M.~Aizenman and G. R.~Grimmett. 
\newblock Strict monotonicity for critical points in percolation 
and ferromagnetic models.
\newblock {\em J. Stat. Phys.} 63:817--835, 1991.

\bibitem{AizenJung07}
M.~Aizenman and P.~Jung.
\newblock On the critical behavior at the lower phase transition of the contact
  process.
\newblock {\em ALEA Lat. Am. J. Probab. Math. Stat.}, 3:301--320, 2007.

\bibitem{bal-boll-riord}
P.~Balister, B.~Bollob\'as  and O.~Riordan.
\newblock Essential enhancements revisited. 
\newblock arXiv:1402.0834

\bibitem{Berg11}
J.~van~den Berg.
\newblock Sharpness of the percolation transition in the two-dimensional
  contact process.
\newblock {\em Ann. Appl. Probab.}, 21(1):374--395, 2011.

\bibitem{BeBrVa08}
J.~van~den Berg, R. Brouwer and B. V\'agv\"olgyi. 
\newblock Box-crossings and continuity results for 
self-destructive percolation in the plane. 
\newblock In: {In and Out of Equilibrium 2} (eds. V. Sidoravicius and
M.-E. Vares), series {\it Progress in Probability} (volume 60),
117--135, Birkh\"auser, 
2008.

\bibitem{BolloRiord06}
B.~Bollob{\'a}s and O.~Riordan.
\newblock The critical probability for random {V}oronoi percolation in the
  plane is 1/2.
\newblock {\em Probab. Theory Related Fields}, 136(3):417--468, 2006.


\bibitem{BolloRiord08}  
B.~Bollob\'as and O.~Riordan, 
\newblock Percolation on random Johnson--Mehl tessellations 
and related models.
\newblock {\it Probab. Theory Related Fields},  140(3): 319--343, 2008.

\bibitem{BoRi06}
B.~Bollob\'as and O.~Riordan,
\newblock {\em Percolation},
Cambridge University Press, 2006.

\bibitem{Broma07}
E.~I. Broman.
\newblock Stochastic domination for a hidden {M}arkov chain with applications
  to the contact process in a randomly evolving environment.
\newblock {\em Ann. Probab.}, 35(6):2263--2293, 2007.

\bibitem{CoxSchin09}
J.~T. Cox and R.~B Schinazi.
\newblock Survival and coexistence for a multitype contact process.
\newblock {\em Ann. Probab.}, 37(3):853--876, 2009.

\bibitem{DurreNeuha97}
R.~Durrett and C.~Neuhauser.
\newblock Coexistence results for some competition models.
\newblock {\em Ann. Appl. Probab.}, 7(1):10--45, 1997.

\bibitem{DurreSwind91}
R.~Durrett and G.~Swindle.
\newblock Are there bushes in a forest?
\newblock {\em Stochastic Process. Appl.}, 37(1):19--31, 1991.

\bibitem{FriedKalai96}
E.~Friedgut and G.~Kalai.
\newblock Every monotone graph property has a sharp threshold.
\newblock {\em Proc. Amer. Math. Soc.}, 124(10):2993--3002, 1996.

\bibitem{Gr99}
G.R.~Grimmett.
\newblock {\em Percolation},
volume 321 of {\em Grundlehren der Mathematischen Wissenschaften}.
Springer, 1999.

\bibitem{kkl}
J. Kahn, G. Kalai and N. Linial.
\newblock The influence of variables on Boolean functions.
\newblock {\em Proc. 29-th Annual Symposium on Foundations of Computer
  Science}, 68--80, Computer Society Press, 1988.


\bibitem{anti-kefi-comment}
S.~K{\'e}fi, C.~L.~Alados, R.~C.~G.~Chaves, Y.~Pueyo, and M.~Rietkerk.
\newblock Is the patch size distribution of 
vegetation a suitable indicator of 
desertification processes? Comment.
\newblock {\em Ecology}, 91(12): 3739--3742, 2010. 

\bibitem{KefiRietkAladoPueyo07}
S.~K{\'{e}}fi, M.~Rietkerk, C.~L. Alados, Y.~Pueyo, V.~P. Papanastasis,
  A.~ElAich, and P.~C.~de Ruiter.
\newblock Spatial vegetation patterns and imminent desertification in
  {M}editerranean arid ecosystems.
\newblock {\em Nature}, 449:213--217, 2007.

\bibitem{kefi11}
S.~K{\'e}fi, M.~Rietkerk, M.~Roy, 
A.~Franc, P.~C.~de~Ruiter and 
M.~Pascual.
\newblock Robust scaling in ecosystems and the 
meltdown of patch size distributions before extinction.
\newblock {\em Ecology Letters}, 14: 29--35, 2011. 

\bibitem{KonnoSchinTanem04}
N.~Konno, R.~B. Schinazi, and H.~Tanemura.
\newblock Coexistence results for a spatial stochastic epidemic model.
\newblock {\em Markov Process. Related Fields}, 10(2):367--376, 2004.

\bibitem{Kucze89}
T.~Kuczek.
\newblock The central limit theorem for the right edge of supercritical
  oriented percolation.
\newblock {\em Ann. Probab.}, 17(4):1322--1332, 1989.

\bibitem{liggett85}
T.~M. Liggett. 
\newblock {\em Interacting Particle Systems}, 
volume 276 of {\em Grundlehren der Mathematischen Wissenschaften}.
Springer, Berlin, 1985.

\bibitem{Ligge99}
T.~M. Liggett.
\newblock {\em Stochastic interacting systems: contact, voter and exclusion
  processes}, volume 324 of {\em Grundlehren der Mathematischen Wissenschaften
  [Fundamental Principles of Mathematical Sciences]}.
\newblock Springer-Verlag, Berlin, 1999.

\bibitem{LiggeSteif06}
T.~M. Liggett and J.~E. Steif.
\newblock Stochastic domination: the contact process, {I}sing models and {FKG}
  measures.
\newblock {\em Ann. Inst. H. Poincar\'e Probab. Statist.}, 42(2):223--243,
  2006.

\bibitem{anti-kefi}
F.~T.~Maestre and A.~Escudero. 
\newblock Is the patch size distribution of vegetation a 
suitable indicator of desertification processes?
\newblock {\em Ecology}, 90(7): 1729--1735, 2009.

\bibitem{anti-kefi-reply}
F.~T.~Maestre and A.~Escudero. 
\newblock Is the patch size distribution of vegetation 
a suitable indicator of desertification processes? Reply.
\newblock {\em Ecology}, 91(12): 3742--3745, 2010. 

\bibitem{Neuha92}
C.~Neuhauser.
\newblock Ergodic theorems for the multitype contact process.
\newblock {\em Probab. Theory Related Fields}, 91(3-4):467--506, 1992.

\bibitem{Remen08}
D.~Remenik.
\newblock The contact process in a dynamic random environment.
\newblock {\em Ann. Appl. Probab.}, 18(6):2392--2420, 2008.

\bibitem{Talag94}
M.~Talagrand.
\newblock On {R}usso's approximate zero-one law.
\newblock {\em Ann. Probab.}, 22(3):1576--1587, 1994.



\end{thebibliography}

\end{document}